\pdfoutput=1
\documentclass[a4paper]{amsart}
\usepackage[T1]{fontenc}
\usepackage{amstext}
\usepackage{amsthm}
\usepackage{amssymb}
\usepackage[unicode=true,pdfusetitle,
 bookmarks=true,bookmarksnumbered=false,bookmarksopen=false,
 breaklinks=false,pdfborder={0 0 0},pdfborderstyle={},backref=false,colorlinks=false]
 {hyperref}

\makeatletter

\numberwithin{equation}{section}
\numberwithin{figure}{section}
\theoremstyle{plain}
\newtheorem{thm}{\protect\theoremname}[section]
\theoremstyle{remark}
\newtheorem{rem}[thm]{\protect\remarkname}
\theoremstyle{definition}
\newtheorem{construction}[thm]{\protect\constructionname}
\theoremstyle{plain}
\newtheorem{prop}[thm]{\protect\propositionname}
\theoremstyle{definition}
\newtheorem{defn}[thm]{\protect\definitionname}
\theoremstyle{plain}
\newtheorem{lem}[thm]{\protect\lemmaname}
\theoremstyle{plain}
\newtheorem{cor}[thm]{\protect\corollaryname}
\theoremstyle{definition}
\newtheorem{warning}[thm]{\protect\warningname}
\theoremstyle{definition}
\newtheorem{example}[thm]{\protect\examplename}

\usepackage{tikz-cd}
\usepackage{enumitem}
\tikzcdset{scale cd/.style={every label/.append style={scale=#1},
    cells={nodes={scale=#1}}}}
\usepackage{quiver}
\usepackage{eucal}
\usepackage{mleftright}
\mleftright

\makeatother

\providecommand{\constructionname}{Construction}
\providecommand{\corollaryname}{Corollary}
\providecommand{\definitionname}{Definition}
\providecommand{\examplename}{Example}
\providecommand{\lemmaname}{Lemma}
\providecommand{\propositionname}{Proposition}
\providecommand{\remarkname}{Remark}
\providecommand{\theoremname}{Theorem}
\providecommand{\warningname}{Warning}

\begin{document}
\global\long\def\sf#1{\mathsf{#1}}%

\global\long\def\cal#1{\mathcal{#1}}%

\global\long\def\bb#1{\mathbb{#1}}%

\global\long\def\bf#1{\mathbf{#1}}%

\global\long\def\frak#1{\mathfrak{#1}}%

\global\long\def\fr#1{\mathfrak{#1}}%

\global\long\def\inp#1{\left\langle #1\right\rangle }%

\global\long\def\pr#1{\left(#1\right)}%

\global\long\def\norm#1{\left\Vert #1\right\Vert }%

\global\long\def\hat#1{\widehat{#1}}%

\global\long\def\opn#1{\operatorname{#1}}%

\global\long\def\Set{\sf{Set}}%

\global\long\def\SS{\sf{sSet}}%

\global\long\def\Cat{\mathcal{C}\sf{at}}%

\global\long\def\Fib{\mathcal{F}\mathsf{ib}}%

\global\long\def\Bund{\mathcal{B}\mathsf{und}}%

\global\long\def\Op{\mathcal{O}\mathsf{p}}%

\global\long\def\Mfld{\cal M\mathsf{fld}}%

\global\long\def\Disk{\cal D\mathsf{isk}}%

\global\long\def\Fin{\sf{Fin}}%

\global\long\def\ot{\leftarrow}%

\global\long\def\Hom{\operatorname{Hom}}%

\global\long\def\lim{\operatorname{lim}}%

\global\long\def\colim{\operatorname{colim}}%

\global\long\def\Map{\operatorname{Map}}%

\global\long\def\Sing{\operatorname{Sing}}%

\global\long\def\Fun{\operatorname{Fun}}%

\global\long\def\Emb{\operatorname{Emb}}%

\global\long\def\Open{\operatorname{Open}}%

\global\long\def\Homeo{\operatorname{Homeo}}%

\global\long\def\Alg{\operatorname{Alg}}%

\global\long\def\id{\mathrm{id}}%

\global\long\def\act{\mathrm{act}}%

\global\long\def\Un{\opn{Un}}%

\global\long\def\op{\mathrm{op}}%

\global\long\def\Comm{\mathrm{Comm}}%

\global\long\def\Top{\mathrm{Top}}%

\global\long\def\t{\otimes}%

\global\long\def\rcone{\triangleright}%

\global\long\def\lcone{\triangleleft}%

\global\long\def\p{\prime}%

\title{Universal Properties of Variations of the Little Cubes Operads}
\author{Kensuke Arakawa}
\begin{abstract}
Given a map $B\to B\Top\pr n$ of spaces, one can define a version
$\bb E_{B}$ of the little cubes operad, whose construction is due
to Lurie. We show that $\bb E_{B}$ enjoys the universal property
that, for every $\infty$-operad $\cal O$, an operad map $\bb E_{B}\to\cal O$
is equivalent to a $\Top\pr n$-equivariant map $B\times_{B\Top\pr n}E\Top\pr n\to\Map\pr{\bb E_{n},\cal O}$.
This gives us an explicit diagram exhibiting $\bb E_{B}$ as a colimit
of $\bb E_{n}$ parametrized by $B$. It also shows that locally constant
factorization algebras satisfy descent, reproving a recent theorem
of Matsuoka.
\end{abstract}

\address{Department of Mathematics, Kyoto University, Kyoto, 606-8502, Japan}
\email{arakawa.kensuke.22c@st.kyoto-u.ac.jp}
\keywords{$\infty$-operads, little cubes, Kan extensions}
\subjclass[2020]{18M60, 18N60, 55U40}
\maketitle

\section{\label{sec:intro}Introduction}

The operad of little $n$-cubes governs homotopy coherent multiplications
by using rectilinear embeddings of cubes. They were first introduced
in the works of Boardman--Vogt \cite{BV73} and May \cite{MayOperad}
to study algebraic structures of iterated loop spaces, and since then,
they have repeatedly appeared in many contexts, from mathematical
physics to embedding calculus. (See \cite{Benoit20} for a survey.) 

As contexts shifted, various modifications of little $n$-cubes operads
emerged. The framed little $n$-disks operad $f\cal D_{n}$ \cite{Getz94, Wahl01},
in which we replace cubes with disks and allow rotations of disks,
is one such example. To unify these variations, Markl and Wahl independently
arrived at the notion of \textit{semidirect products} of operads \cite{Markl99, Wahl01}:
Given a topological operad $\cal O$ equipped with a (left) action
of a topological monoid $M$, the semidirect product $M\ltimes\cal O$
is defined by setting $\pr{M\ltimes\cal O}\pr k=M^{k}\times\cal O\pr k$,
with obvious structure maps. For example, the special orthogonal group
${\rm SO}\pr n$ acts on the operad $\cal D_{n}$ of little $n$-cubes
by moving the centers of the small disks (with their radii fixed)
by using the action of ${\rm SO}\pr n$ on $\bb R^{n}$, and the resulting
operad ${\rm SO}\pr n\ltimes\cal D_{n}$ is nothing but the framed
little $n$-disks operad. If we want to be more restrictive on the
class of embeddings, we can just choose a subgroup $H\subset{\rm SO}\pr n$
(or more generally, a group homomorphism $H\to{\rm SO}\pr n$) and
form the semidirect product $H\ltimes\cal D_{n}$. 

In \cite{DAGVI}, Lurie introduced an $\infty$-operadic analog of
these variations, which we now recall.
\begin{rem}
The rest of this paper relies heavily on the theory of $\infty$-categories
and $\infty$-operads. However, for the most part, one can get the
idea of this section just by replacing these gadgets with topological
categories and colored topological operads and by replacing various
$\infty$-categorical constructions (limits and colimits, Kan extensions,
sheaves, etc) by the classical derived constructions (homotopy limits
and homotopy colimits, homotopy Kan extensions, homotopy sheaves,
etc). Many $\infty$-operads have $\otimes$ (and sometimes $\amalg$)
in their exponents, but this is just a notational convention. Remark
\ref{rem:main} might also help to understand the main result.
\end{rem}

\begin{construction}
Given a map $B\to B{\rm SO}\pr n$ of spaces, we define an $\infty$-operad
$\bb E_{B}^{\t}$ to be the pullback of the diagram
\[
N\pr{f\cal D_{n}}^{\t}\to B{\rm SO}\pr n^{\amalg}\ot B^{\amalg}
\]
in the $\infty$-category $\Op_{\infty}$ of $\infty$-operads, where:
\begin{itemize}
\item $N\pr -^{\t}$ denotes the operadic nerve functor, which converts
a simplicial operad into an $\infty$-operad;
\item $B^{\amalg}$ denotes the colimit of the constant diagram $\Comm^{\t}:B\to\Op_{\infty}$
at the commutative $\infty$-operad.\footnote{This description of $B^{\amalg}$ is not immediate from Lurie's definition,
but we will see that it is correct (Lemma \ref{lem:key}).} In the case where $B=B{\rm SO}\pr n$, a direct computation shows
that $N\pr{{\rm SO}\pr n\ltimes\Comm}^{\t}=B{\rm SO}\pr n^{\amalg}$,
and the left hand map is induced by the map $f\cal D_{n}={\rm SO}\pr n\ltimes\cal D_{n}\to{\rm SO}\pr n\ltimes\Comm$.
\end{itemize}
In fact, Lurie's definition is more general, as he defines the $\infty$-operad
$\bb E_{B}^{\t}$ for any map $B\to B\Top\pr n$ of spaces. Instead
of spelling out the details here, we refer the readers to Subsection
\ref{subsec:def}.
\end{construction}

For example, given a homomorphism $H\to{\rm SO}\pr n$ of topological
groups, the $\infty$-operad $\bb E_{BH}^{\t}$ is equivalent to $N\pr{H\ltimes\cal D_{n}}^{\t}$. 

Notice that the above construction differs from the classical one
in one crucial respect: Lurie's definition starts with the framed
$n$-disks operad, instead of the action of ${\rm SO}\pr n$ on the
little $n$-disks operads. The reason for this is that the operadic
nerve functor $N\pr -^{\t}$ is not enriched (at least not in an obvious
way), so it is difficult to construct a diagram $B{\rm SO}\pr n\to\Op_{\infty}$
encoding the action of ${\rm SO}\pr n$ on $N\pr{\cal D_{n}}^{\t}$.

Now the classical notion of semidirect products of groups is a special
case of the Grothendieck construction, which in turn is a special
case of colimits. This suggests, as Lurie himself remarked informally
in \cite[Remark 3.1.10]{DAGVI}, that the $\infty$-operad $\bb E_{B}^{\t}$
is the colimit of some diagram $B\to\Op_{\infty}$ exhibiting an action
of $B$ on $\bb E_{n}^{\t}$. This is indeed true, although formal
proofs had not appeared until quite recently (see, e.g., \cite[Proposition 2.2]{HKK22}
and \cite[Remark 2.2]{DHLSW23}). 

However, for the exact same reason as we had to start with the semidirect
product and not from the action, it seems difficult to pinpoint exactly
\textit{which} diagram's colimit $\bb E_{B}^{\t}$ is. Because of
this, the description of $\bb E_{B}$-algebras has always been opaque,
camouflaged with phrases like ``$\bb E_{n}$-algebras with coherent
actions of $B$.'' In this paper, we improve this situation by constructing
an explicit diagram $B\to\Op_{\infty}$ whose colimit is $\bb E_{B}^{\t}$,
thereby making the universal property of $\bb E_{B}^{\t}$ lucid.
We will achieve this by showing that the functor $\bb E_{\bullet}^{\t}$
(not just each $\bb E_{B}^{\t}$) enjoys a certain universal property.

As a motivation, consider the universal case where $B=B\Top\pr n$.
Our task is to find an $\infty$-operad which is equivalent to $\bb E_{n}^{\t}$
and admits a $\Top\pr n$-action. This might be difficult if we are
stuck with the picture of little disks, but there is in fact a canonical
one: Consider the classifying map $\bb R^{n}\to B\Top\pr n$ of the
tangent microbundle. Since the classifying maps of tangent microbundles
are compatible with embeddings, the object $\bb R^{n}\in\cal S_{/B\Top\pr n}$
has a $\Top\pr n$-action. Since the construction of $\bb E_{\bullet}^{\t}$
is functorial, the $\infty$-operad $\bb E_{\bb R^{n}}^{\t}$ inherits
a $\Top\pr n$-action. On the other hand, since $\bb R^{n}$ is contractible,
the $\infty$-operad $\bb E_{\bb R^{n}}^{\t}$ is equivalent to $\bb E_{{\rm pt}}^{\t}\simeq\bb E_{n}^{\t}$.
Thus we obtain an $\infty$-operad equivalent to $\bb E_{n}^{\t}$,
equipped with a $\Top\pr n$-action. We will confirm that this is
the action we were after:
\begin{prop}
\label{prop:intro}For any map $B\to B\Top\pr n$ of spaces, the $\infty$-operad
$\bb E_{B}^{\t}$ is the colimit of the composite
\[
B\to B\Top\pr n\xrightarrow{\bb R^{n}}\cal S_{/B\Top\pr n}\xrightarrow{\bb E_{\bullet}^{\t}}\Op_{\infty}.
\]
\end{prop}

We can interpret Proposition \ref{prop:intro} as describing the \textit{local
}universal property of $\bb E_{B}^{\t}$, in the sense that $B$ is
fixed. We will deduce this local property from the following \textit{global
}universal property of $\bb E_{\bullet}^{\t}$, which is the main
result of this paper:
\begin{thm}
[Theorem \ref{thm:main}]\label{thm:intro1}The diagram % https://q.uiver.app/#q=WzAsMyxbMiwwLCJcXG1hdGhjYWx7T31cXG1hdGhzZntwfV9cXGluZnR5Il0sWzAsMCwiXFxvcGVyYXRvcm5hbWV7RnVufShCXFxtYXRocm17VG9wfShuKV57XFxtYXRocm17b3B9fSxcXG1hdGhjYWx7U30pIl0sWzEsMSwiXFxtYXRoY2Fse1N9X3svQlxcbWF0aHJte1RvcH0obil9Il0sWzIsMCwiXFxtYXRoYmJ7RX1fXFxidWxsZXReXFxvdGltZXMgIiwyXSxbMSwwLCItXFxvdGltZXNfe1xcbWF0aHJte1RvcH0obil9IFxcbWF0aGJie0V9X3tcXG1hdGhiYntSfV5ufV5cXG90aW1lcyJdLFsxLDIsIlxcc2ltZXEiLDJdXQ==
\[\begin{tikzcd}[row sep = small]
	{\operatorname{Fun}(B\mathrm{Top}(n)^{\mathrm{op}},\mathcal{S})} && {\mathcal{O}\mathsf{p}_\infty} \\
	& {\mathcal{S}_{/B\mathrm{Top}(n)}}
	\arrow["{-\otimes_{\mathrm{Top}(n)} \mathbb{E}_{\mathbb{R}^n}^\otimes}", from=1-1, to=1-3]
	\arrow["\simeq"', from=1-1, to=2-2]
	\arrow["{\mathbb{E}_\bullet^\otimes }"', from=2-2, to=1-3]
\end{tikzcd}\]commutes up to natural equivalence, where the left slanted arrow is
the straightening--unstraightening equivalence and $-\otimes_{\Top\pr n}\bb E_{\bb R^{n}}^{\t}$
is the left Kan extension of the functor $\bb E_{\bb R^{n}}^{\t}:B\Top\pr n\to\Op_{\infty}$
along the Yoneda embedding. 
\end{thm}

In other words, for any right $\Top\pr n$-space $X$ with classifying
map $B\to B\Top\pr n$, there is a natural equivalence of $\infty$-operads

\[
X\otimes_{\Top\pr n}\bb E_{\bb R^{n}}^{\otimes}\simeq\bb E_{B}^{\t}.
\]

To explain why Theorem \ref{thm:intro1} implies Proposition \ref{prop:intro},
let $X\in\Fun\pr{B\Top\pr n^{\op},\cal S}$ be a right $\Top\pr n$-space
with classifying map $B\to B\Top\pr n$. By the colimit formula for
Kan extensions, we have
\[
X\otimes_{B\Top\pr n}\bb E_{\bb R^{n}}^{\t}=\colim\pr{B\Top\pr n_{/X}\to B\Top\pr n\xrightarrow{\bb E_{\bb R^{n}}^{\t}}\Op_{\infty}}.
\]
The Yoneda lemma implies that the right fibration $B\Top\pr n_{/X}\to B\Top\pr n$
classifies the right $\Top\pr n$-space $X$, so $B$ is equivalent
to $B\Top\pr n_{/X}$ as a space over $B\Top\pr n$. Also, Theorem
\ref{thm:intro1} says that $X\otimes_{B\Top\pr n}\bb E_{\bb R^{n}}^{\t}$
is equivalent to $\bb E_{B}^{\t}$. Hence $\bb E_{B}^{\t}$ is the
colimit of the diagram $B\to B\Top\pr n\xrightarrow{\bb E_{\bb R^{n}}^{\t}}\Op_{\infty}$,
as claimed.
\begin{rem}
The above computation also explains why Theorem \ref{thm:intro1}
ought to be true. Indeed, if $f:B\Top\pr n\to\Op_{\infty}$ is a diagram
whose colimit is $\bb E_{B\Top\pr n}^{\t}$, then $\bb E_{B}^{\t}$
should be the colimit of the composite $B\to B\Top\pr n\xrightarrow{f}\Op_{\infty}$.
The above computation then shows that the composite $\Fun\pr{B\Top\pr n^{\op},\cal S}\simeq\cal S_{/B\Top\pr n}\xrightarrow{\bb E_{\bullet}^{\t}}\Op_{\infty}$
should be the left Kan extension of its restriction along the Yoneda
embedding $y:B\Top\pr n\to\Fun\pr{B\Top\pr n^{\op},\cal S}$. It is
not hard to see that the composite 
\[
B\Top\pr n\xrightarrow{y}\Fun\pr{B\Top\pr n^{\op},\cal S}\simeq\cal S_{/B\Top\pr n}\xrightarrow{\bb E_{\bullet}^{\t}}\Op_{\infty}
\]
is nothing but the functor $\bb E_{\bb R^{n}}^{\t}:B\Top\pr n\to\Op_{\infty}$.
Thus we arrive at the statement of Theorem \ref{thm:intro1}. 
\end{rem}

\begin{rem}
The action of $\Top\pr n$ on $\bb E_{\bb R^{n}}^{\t}$ admits a nice
geometric picture. Given an $n$-manifold $M$, regarded as a space
over $B\Top\pr n$ by the classifying map of its tangent microbundle,
the $\infty$-operad $\bb E_{M}^{\t}$ can informally be described
as follows\footnote{Recall that, despite the name, $\infty$-operads are a generalization
of colored operads.}:
\begin{enumerate}
\item The colors of $\bb E_{M}^{\t}$ are the embeddings $\bb R^{n}\to M$.
\item Given a collection $\iota_{1},\dots,\iota_{k},\iota:\bb R^{n}\to M$
of embeddings, a multiarrow $\pr{\iota_{1},\dots,\iota_{k}}\to\iota$
consists of an embedding $f:\coprod_{i=1}^{k}\bb R^{n}\to\bb R^{n}$
and a collection of isotopies $\{\iota_{i}\simeq\iota f\}_{1\leq i\leq k}$.
\end{enumerate}
In other words, the $\infty$-operad $\bb E_{M}^{\t}$ is something
like the little $n$-cubes operad, but in which the little $n$-cubes
are now embedded into the ``ambient space'' $M$. Every embedding
$M\to N$ of $n$-manifolds induces a map $\bb E_{M}^{\t}\to\bb E_{N}^{\t}$
which embeds the little cubes into the larger ambient space $N$.
The action of $\Top\pr n$ on $\bb E_{\bb R^{n}}^{\t}$ arises when
we take $M=N=\bb R^{n}$.
\end{rem}

\begin{rem}
\label{rem:main}Although this is just a paraphrase of Theorem \ref{thm:main},
the following reformulation of the universal property of $\bb E_{B}^{\t}$
is worth pointing out: For every $\infty$-operad $\cal O^{\t}$,
a map $\bb E_{B}^{\t}\to\cal O^{\t}$ of $\infty$-operads is equivalent
to a $\Top\pr n$-equivariant map $B\times_{B\Top\pr n}E\Top\pr n\to\Map_{\Op_{\infty}}\pr{\bb E_{\bb R_{n}}^{\t},\cal O^{\t}}$.
More precisely, there is a homotopy equivalence
\[
\Map_{\Op_{\infty}}\pr{\bb E_{B}^{\t},\cal O^{\t}}\simeq\Map_{\Fun\pr{B\Top\pr n^{\op},\cal S}}\pr{B\times_{B\Top\pr n}E\Top\pr n,\Map_{\Op_{\infty}}\pr{\bb E_{\bb R^{n}}^{\t},\cal O^{\t}}}
\]
which is natural in $\cal O^{\t}\in\Op_{\infty}$ and $B\in\cal S_{/B\Top\pr n}$.

For instance, if $M$ is an $n$-manifold, then a map $\bb E_{M}^{\t}\to\cal O^{\t}$
is equivalent to a $\Top\pr n$-equivariant map from the topological
frame bundle of $M$ (i.e., the principal $\Top\pr n$-bundle associated
with the tangent microbundle) to $\Map_{\Op_{\infty}}\pr{\bb E_{\bb R^{n}}^{\t},\cal O^{\t}}$. 
\end{rem}

\subsection*{Related Works}

Algebras over the $\infty$-operad $\bb E_{M}^{\t}$, where $M$ is
an $n$-manifold, are called\textit{ locally constant factorization
algebras}\footnote{More precisely, $\bb E_{M}$-algebras are equivalent to locally constant
factorization algebras; see \cite[Theorem 5.4.5.9]{HA}.}\textit{ }on $M$ and appear in classical and quantum field theory
\cite{FAQFT_1, FAQFT_2}. Theorem \ref{thm:intro1} has an important
antecedent from the theory of\textbf{ }locally constant factorization
algebras, as we now explain.

A basic problem with anything associated with manifolds is whether
it obeys the local-to-global principle. Matsuoka showed that locally
constant factorization algebras have this property:
\begin{thm}
[Matsuoka {\cite[Theorem 1.3]{Matsu17}}] Let $\cal C^{\t}$ be a
symmetric monoidal $\infty$-category and let $M$ be an $n$-manifold.
The assignment
\[
U\mapsto\Alg_{\bb E_{U}}\pr{\cal C}
\]
determines a sheaf of $\infty$-categories on $M$, where $\Alg_{\bb E_{U}}\pr{\cal C}$
denotes the $\infty$-category of $\bb E_{U}$-algebras in $\cal C$.
\end{thm}

Matsuoka's theorem is a special case of Theorem \ref{thm:intro1}.
To see this, recall that if $\cal C$ is a small $\infty$-category
and $\cal D$ is an $\infty$-category with small colimits, then a
functor $F:\Fun\pr{\cal C^{\op},\cal S}\to\cal D$ is a left Kan extension
of its restriction along the Yoneda embedding if and only if it preserves
small colimits. In particular, Theorem \ref{thm:intro1} implies that
the functor $\bb E_{\bullet}^{\t}:\cal S_{/B\Top\pr n}\to\Op_{\infty}$
preserves small colimits. Since the tangent classifier functor $\tau:\Open\pr M\to\cal S_{/B\Top\pr n}$
is evidently a cosheaf,\footnote{For a formal proof, see the proof of Corollary \ref{cor:FA_glue}.}
this means that the composite
\[
\Open\pr M\xrightarrow{\tau}\cal S_{/B\Top\pr n}\xrightarrow{\bb E_{\bullet}^{\t}}\Op_{\infty}
\]
is also a cosheaf. Matsuoka's theorem then follows from the observation
that the functor $\Alg_{\bullet}\pr{\cal C}:\Op_{\infty}^{\op}\to\Cat_{\infty}$
preserves small limits. \iffalse (it preserves pullbacks and products)\fi In
this sense, Theorem \ref{thm:intro1} generalizes Matsuoka's theorem
by allowing gluing of arbitrary spaces over $B\Top\pr n$, not just
manifolds.\footnote{Another direction of generalization of Matsuoka's theorem is investigated
in \cite{KSW24}, where they consider the gluing properties of constructible
factorization algebras on stratified manifolds.} 

We should also remark that our technique is quite robust and applies
equally well to general topological operads with actions of group-like
topological monoids. We will work with the little $n$-cubes operads
for concreteness, but readers interested in a version of Theorem \ref{thm:intro1}
for other operads should have no problem translating the arguments
of this paper (see Subsection \ref{subsec:main}) to meet their needs.

Finally, algebras over $\bb E_{B}^{\t}$ play a central role in the
theory of \textit{topological chiral homology} \cite[$\S$5.5]{HA},
alias \textit{factorization homology} \cite{FHTM} (where they go
under the name of $\Disk_{n}^{B}$-algebras). We hope that this paper
adds to the conceptual understanding of $\bb E_{B}^{\t}$-algebras.

\iffalse

It is natural to ask whether we can describe the colimit of an arbitrary
diagram $f:\cal I\to\Op_{\infty}$ that carry all objects of $\cal I$
to the little $n$-cubes operads. If $f$ factors through the map
$\bb E_{\bb R^{n}}^{\otimes}:B\Top\pr n\to\Op_{\infty}$, then the
colimit of $f$ is equivalent to $\bb E_{B\cal I}^{\t}$, where $B\cal I=\cal I[\cal I^{-1}]$
denotes the localization of $\cal I$ with respect to all morphisms.
Surprisingly, such a factorization exists if $n\leq2$, by a theorem
of Horel \cite[Theorem 8.5]{Horel17}. For a general $n$, a necessary
condition for such a factorization to exist is that $f$ factors through
$B\cal I$. A recent result of Horel, Kuranich, and Kupers shows that
this necessary condition is always satisfied \cite[Theorem B]{HKK22}. 

\fi

\subsection*{Outline of the Paper}

In Section \ref{sec:main}, we prove Theorem \ref{thm:intro1}, assuming
a lemma on some construction of $\infty$-operads. The lemma will
then be proved in Section \ref{sec:lemma}.

\subsection*{Acknowledgment}

I am grateful to an anonymous referee for carefully reading a manuscript
of this paper, offering valuable feedback, and providing a simple
proof of Proposition \ref{prop:univ_colim_criterion}. I also thank
Michael Weiss for his comment, which improved the exposition of this
paper. Finally, I appreciate Daisuke Kishimoto and Mitsunobu Tsutaya
for their constant support and encouragement. This work was partially
supported by JSPS KAKENHI Grant Number 24KJ1443.

\subsection*{Notation and Convention}
\begin{itemize}
\item Following \cite{HTT}, we use the term $\infty$-category as a synonym
for Joyal's quasicategory \cite{Joyal_qcat_Kan}. Unless stated otherwise,
our notation and terminology follow Lurie's books \cite{HTT} and
\cite{HA}.
\item If $\bf A$ is a simplicial model category, we let $\bf A^{\circ}\subset\bf A$
denote the full simplicial subcategory spanned by the fibrant-cofibrant
objects.
\item The simplicial category of simplicial sets, equipped with the Kan--Quillen
model structure, will be denoted by $\SS$. 
\item The category of marked simplicial sets will be denoted by $\SS^{+}$. 
\item We will write $\Fin_{\ast}$ for the category whose objects are the
finite pointed sets $\inp n=\pr{\{\ast,1,\dots,n\},\ast}$ and whose
morphisms are the maps of pointed sets.
\item We let $\Op_{\infty}^{\Delta}$ denote the simplicial category of
$\infty$-operads, defined as in \cite[Definition 2.1.4.1]{HA}, and
write $\Op_{\infty}$ for its homotopy coherent nerve.
\item If $M$ and $N$ are topological manifolds (without boundaries), we
will write $\Emb\pr{M,N}$ for the space of topological embeddings
$M\to N$, topologized by the compact-open topology. We let $\Top\pr n\subset\Emb\pr{\bb R^{n},\bb R^{n}}$
denote the subspace of self-homeomorphisms of $\bb R^{n}$.
\item We will write $\Mfld_{n}^{\Delta}$ for the simplicial category whose
objects are the topological $n$-manifolds (without boundary), and
whose hom-simplicial sets are given by $\Sing\Emb\pr{-,-}$. The homotopy
coherent nerve of $\Mfld_{n}^{\Delta}$ will be denoted by $\Mfld_{n}$.
\item We will write $B\Top\pr n^{\Delta}\subset\Mfld_{n}^{\Delta}$ for
the full simplicial subcategory spanned by $\bb R^{n}$, and let $B\Top\pr n$
denote its homotopy coherent nerve. (This notation is justified by
Kister--Mazur's theorem \cite[Theorem 1]{Kister_MFB}, which says
that the inclusion $\Top\pr n\hookrightarrow$ $\Emb\pr{\bb R^{n},\bb R^{n}}$
is a homotopy equivalence. We also remark that their theorem implies
that $B\Top\pr n$ is a Kan complex.)
\item If $K$ is a simplicial set, we will denote the cone point of the
simplicial set $K^{\rcone}$ by $\infty$.
\item If $\cal C$ is an $\infty$-category, we will write $\cal C^{\simeq}\subset\cal C$
for the maximal sub Kan complex of $\cal C$.
\end{itemize}

\section{\label{sec:main}Proof of Theorem \ref{thm:intro1}}

The goal of this section is to prove Theorem \ref{thm:intro1}. We
start by giving precise definitions of the functor $\bb E_{\bullet}^{\t}$
and related constructions in Subsection \ref{subsec:def}. We then
prove some preliminary results in Subsection \ref{subsec:univ_colim}.
After this, we prove the theorem in Subsection \ref{subsec:main},
assuming a lemma that we will prove in Section \ref{sec:lemma} (Lemma
\ref{lem:key}). In Subsection \ref{subsec:glue}, we use the theorem
to give an alternative proof of Matsuoka's gluing theorem on locally
constant factorization algebras.

\subsection{\label{subsec:def}Definitions}

In this subsection, we recall the definitions of the functor $\bb E_{\bullet}^{\t}:\cal S_{/B\Top\pr n}\to\Op_{\infty}$
and some related constructions. 
\begin{defn}
\cite[Construction 2.4.3.1]{HA} We define a simplicial functor $\pr -^{\amalg}:\SS^{\circ}\to\Op_{\infty}^{\Delta}$
as follows: Let $\iota:\Fin_{\ast}\to\Set$ denote the forgetful functor
and let $\int\iota$ denote its category of elements. We let $\Gamma^{*}\subset\int\iota$
denote the full subcategory spanned by the objects $\pr{\inp k,i}$
such that $i\in\inp k\setminus\{\ast\}$. We then define $\pr -^{\amalg}:\SS\to\SS_{/N\pr{\Fin_{\ast}}}$
as the right adjoint of the simplicial functor $X\mapsto X\times_{N\pr{\Fin_{\ast}}}N\pr{\Gamma^{*}}$.
According to \cite[Proposition 2.4.3.3]{HA}, this simplicial functor
lifts to a simplicial functor $\pr -^{\amalg}:\SS^{\circ}\to\Op_{\infty}^{\Delta}$. 
\end{defn}

\begin{rem}
Let $\cal C$ be a simplicial category. The simplicial set $N\pr{\cal C}^{\amalg}$
is isomorphic to the homotopy coherent nerve of the simplicial category
$\cal C^{\amalg}$ whose objects are the (possibly empty) sequences
$\pr{C_{1},\dots,C_{k}}$ of objects of $\cal C$, and whose hom-simplicial
sets are given by
\[
\cal C^{\amalg}\pr{\pr{C_{1},\dots,C_{k}},\pr{D_{1},\dots,D_{l}}}=\coprod_{\alpha:\inp k\to\inp l}\prod_{j=1}^{l}\prod_{i\in\alpha^{-1}\pr j}\cal C\pr{C_{i},D_{j}},
\]
where the coproduct ranges over the morphisms $\inp k\to\inp l$ in
$\Fin_{\ast}$.
\end{rem}

\begin{rem}
\label{rem:amalg}Using \cite[Remark B.3.9]{HA}, we can check that
the functor $N\pr{\Gamma^{*}}\to N\pr{\Fin_{\ast}}$ is a flat inner
fibration. Therefore, the functor $-\times_{N\pr{\Fin_{\ast}}}N\pr{\Gamma^{*}}:\SS_{/N\pr{\Fin_{\ast}}}\to\SS$
is left Quillen with respect to the Joyal model structures \cite[Corollary B.3.15]{HA}.
It follows that every Kan fibration $X\to Y$ of Kan complexes induces
a fibration of $\infty$-operads $X^{\amalg}\to Y^{\amalg}$.
\end{rem}

\begin{defn}
\label{def:E_B}\cite[Definition 5.4.2.10]{HA} We define a simplicial
functor $\bb E_{\bullet}^{\t}:\pr{\SS_{/B\Top\pr n}}^{\circ}\to\Op_{\infty}^{\Delta}$
as follows: Let $\bb E_{B\Top\pr n}^{\t}$ denote the operadic nerve
of the topological operad whose space of operations of arity $k$
is $\Emb\pr{\bb R^{n}\times\{1,\dots,k\},\bb R^{n}}$. The restriction
map
\[
\Emb\pr{\bb R^{n}\times\{1,\dots,k\},\bb R^{n}}\to\prod_{i=1}^{k}\Emb\pr{\bb R^{n},\bb R^{n}}
\]
determines a functor $\bb E_{B\Top\pr n}^{\t}\to B\Top\pr n^{\amalg}$.
Given a Kan fibration $B\to B\Top\pr n$ of Kan complexes, we set
\[
\bb E_{B}^{\t}=\bb E_{B\Top\pr n}^{\t}\times_{B\Top\pr n^{\amalg}}B^{\amalg}.
\]
Note that $\bb E_{B}^{\t}$ is an $\infty$-operad by Remark \ref{rem:amalg}.
\end{defn}

\begin{defn}
\cite[$\S$2.1]{FHTM} We define the \textbf{tangent classifier functor}
$\tau:\Mfld_{n}^{\Delta}\to\pr{\SS_{/B\Top\pr n}}^{\circ}$ to be
the composite simplicial functor
\begin{align*}
\Mfld_{n}^{\Delta} & \xrightarrow{\text{yoneda}}\Fun^{s}\pr{\pr{\Mfld_{n}^{\Delta}}^{\op},\SS^{\circ}}\\
 & \xrightarrow{\text{restriction}}\Fun^{s}\pr{\pr{B\Top\pr n^{\Delta}}^{\op},\SS^{\circ}}\\
 & \xrightarrow{\Un_{\varepsilon}}\pr{\SS_{/B\Top\pr n}}^{\circ},
\end{align*}
where $\Fun^{s}\pr{-,-}$ denotes the simplicial category of simplicial
functors, and $\Un_{\varepsilon}$ denotes the unstraightening functor
\cite[$\S$2.2.1]{HTT} with respect to the counit map $\varepsilon:\fr C[B\Top\pr n]\to B\Top\pr n^{\Delta}$.
\end{defn}

\begin{rem}
\label{rem:t_cls}The name ``tangent classifier'' comes from the
observation that, if $M$ is an $n$-manifold, then $\tau\pr M$ can
be regarded as a model of the classifying map of the tangent microbundle
of $M$. (In particular, the total space of $\tau\pr M$ has the homotopy
type of $\Sing M$.) More precisely, let $\xi:TM\to M$ denote the
fiber bundle with fiber $\bb R^{n}$ associated to the tangent microbundle
of $M$, and let $\pi:\opn{Fr}\pr M\to M$ denote the associated $\Top_{0}\pr n$-bundle,
where $\Top_{0}\pr n\subset\Top\pr n$ denotes the subgroup of homeomorphisms
that fixes the origin. (Thus, for each point $p\in M$, the space
$T_{p}M=\xi^{-1}\pr p$ is an open subset of $\{p\}\times M$, and
$\opn{Fr}_{p}\pr M=\pi^{-1}\pr p$ is the subspace $\Homeo_{0}\pr{\bb R^{n},T_{p}M}\subset\Emb\pr{\bb R^{n},T_{p}M}$
of homeomorphisms $\bb R^{n}\to T_{p}M$ carrying the origin to $\pr{p,p}\in T_{p}M$.)
We will see in the next paragraph that the map
\[
\theta:\opn{Fr}\pr M\to\Emb\pr{\bb R^{n},M}
\]
is a weak homotopy equivalence. Since $\tau\pr M$ is the classifying
map of $\Sing\Emb\pr{\bb R^{n},M}\in\Fun\pr{B\Top\pr n^{\op},\cal S}$
by definition, and since the inclusion $\Top_{0}\pr n\hookrightarrow\Top\pr n$
is a homotopy equivalence, this proves that $\tau\pr M$ can be identified
with the classifying map of the tangent microbundle of $M$. 

To see that $\theta$ is a weak homotopy equivalence, observe that
$\opn{Fr}\pr M$ and $\opn{Emb}\pr{\bb R^{n},M}$ project to $M$
via Serre fibrations. (For $\opn{Fr}\pr M$, this is clear because
$\pi$ is a fiber bundle. For $\Emb\pr{\bb R^{n},M}$, this is proved
in \cite[Proposition 2.7]{A24b}.) It will therefore suffice to show
that for each $p\in M$, the inclusion
\[
\theta_{p}:\opn{Fr}_{p}\pr M=\Homeo_{0}\pr{\bb R^{n},T_{p}M}\hookrightarrow\Emb\pr{\bb R^{n},M}\times_{M}\{p\}
\]
is a weak homotopy equivalence. We can factor $\theta_{p}$ as
\[
\Homeo_{0}\pr{\bb R^{n},T_{p}M}\xrightarrow{\phi}\Emb\pr{\bb R^{n},T_{p}M}\times_{T_{p}M}\{\pr{p,p}\}\xrightarrow{\psi}\Emb\pr{\bb R^{n},M}\times_{M}\{p\}.
\]
The map $\phi$ is a homotopy equivalence by Kister--Mazur's theorem
\cite[Theorem 1]{Kister_MFB}. The map $\psi$ is a weak homotopy
equivalence by \cite[Proposition 2.9]{A24b}, and we are done.
\end{rem}

\begin{defn}
\label{def:E_M}We define a simplicial functor $\bb E_{\bullet}^{\t}:\Mfld_{n}^{\Delta}\to\Op_{\infty}^{\Delta}$
to be the composite
\[
\Mfld_{n}^{\Delta}\xrightarrow{\tau}\pr{\SS_{/B\Top\pr n}}^{\circ}\xrightarrow{\bb E_{\bullet}^{\t}}\Op_{\infty}^{\Delta}.
\]
The restriction $\bb E_{\bullet}^{\t}\vert B\Top\pr n^{\Delta}$ will
be denoted by $\bb E_{\bb R^{n}}^{\t}:B\Top\pr n^{\Delta}\to\Op_{\infty}^{\Delta}$.
\end{defn}

\begin{rem}
\label{rem:direct_computation}By direct computation, we can check
that if $M$ is an $n$-manifold, then the object $\tau\pr M\in\SS_{/B\Top\pr n}$
is (isomorphic to) the projection $B\Top\pr n_{/M}=B\Top\pr n\times_{\Mfld_{n}}\Mfld_{n/M}\to B\Top\pr n$.
(Readers unfamiliar with the explicit description of the unstraightening
functor should consult \cite[$\S$4]{A24c}.) In particular, our definition
of the $\infty$-operad $\bb E_{M}^{\t}=\bb E_{\tau\pr M}^{\t}$ coincides
with Lurie's definition \cite[Definition 5.4.5.1]{HA}.
\end{rem}

\begin{defn}
\label{def:abuse}We will write $\pr -^{\amalg}:\cal S\to\Op_{\infty}$
for the homotopy coherent nerve of the simplicial functor $\pr -^{\amalg}:\SS^{\circ}\to\Op_{\infty}^{\Delta}$.
We define functors $\bb E_{\bullet}^{\t}:N\pr{\pr{\SS_{/B\Top\pr n}}^{\circ}}\to\Op_{\infty}$,
$\tau:\Mfld_{n}\to N\pr{\pr{\SS_{/B\Top\pr n}}^{\circ}}$, and $\bb E_{\bullet}^{\t}:\Mfld_{n}\to\Op_{\infty}$
similarly.

By slightly abusing notation, we will use the symbol $\bb E_{\bullet}^{\t}$
for any functor $F:\cal S_{/B\Top\pr n}\to\Op_{\infty}$ rendering
the diagram% https://q.uiver.app/#q=WzAsMyxbMCwwLCJOKChcXG1hdGhzZntzU2V0fV97L0JcXG1hdGhybXtUb3B9KG4pfSleXFxjaXJjKSJdLFsyLDAsIlxcbWF0aGNhbHtPfVxcbWF0aHNme3B9X1xcaW5mdHkiXSxbMSwxLCJcXG1hdGhjYWx7U31fey9CXFxtYXRocm17VG9wfShuKX0iXSxbMCwxLCJOKFxcbWF0aGJie0V9X1xcYnVsbGV0IF5cXG90aW1lcyApIl0sWzAsMiwiXFxzaW1lcSIsMl0sWzIsMSwiRiIsMl1d
\[\begin{tikzcd}
	{N((\mathsf{sSet}_{/B\mathrm{Top}(n)})^\circ)} && {\mathcal{O}\mathsf{p}_\infty} \\
	& {\mathcal{S}_{/B\mathrm{Top}(n)}}
	\arrow["{N(\mathbb{E}_\bullet ^\otimes )}", from=1-1, to=1-3]
	\arrow["\simeq"', from=1-1, to=2-2]
	\arrow["F"', from=2-2, to=1-3]
\end{tikzcd}\]commutative up to natural equivalence, where the slanted arrow on
the left is the categorical equivalence of \cite[\href{https://kerodon.net/tag/01ZT}{Tag 01ZT}]{kerodon}. 
\end{defn}

\begin{defn}
\label{def:Un}Let $\cal C^{\Delta}$ be a small simplicial category
whose hom-simplicial sets are Kan complexes, and let $\cal C=N\pr{\cal C^{\Delta}}$
be its homotopy coherent nerve. By the \textbf{unstraightening functor
}(or the\textbf{ unstraightening equivalence}), we mean any functor
$\Fun\pr{\cal C,\cal S}\to N\pr{\pr{\SS_{/\cal C}}_{{\rm contra}}^{\circ}}$
rendering the diagram % https://q.uiver.app/#q=WzAsMyxbMCwwLCJOKFxcb3BlcmF0b3JuYW1le0Z1bn1ecyhcXG1hdGhjYWx7Q31fXFxEZWx0YV5cXG1hdGhybXtvcH0sXFxtYXRoc2Z7c1NldH0pXlxcY2lyYykiXSxbMSwwLCJOKChcXG1hdGhzZntzU2V0fV97L1xcbWF0aGNhbHtDfX0pXlxcY2lyYyBfe1xcbWF0aHJte2NvbnRyYX19KSJdLFswLDEsIlxcb3BlcmF0b3JuYW1le0Z1bn0oXFxtYXRoY2Fse0N9XntcXG1hdGhybXtvcH19LFxcbWF0aGNhbHtTfSkiXSxbMCwxLCJOKFxcb3BlcmF0b3JuYW1le1VufV9cXHZhcmVwc2lsb24pIl0sWzAsMiwiXFxzaW1lcSIsMl0sWzIsMSwiIiwyLHsic3R5bGUiOnsiYm9keSI6eyJuYW1lIjoiZGFzaGVkIn19fV0sWzAsMSwiXFxzaW1lcSIsMl1d
\[\begin{tikzcd}
	{N(\operatorname{Fun}^s(\mathcal{C}_\Delta^\mathrm{op},\mathsf{sSet})^\circ)} & {N((\mathsf{sSet}_{/\mathcal{C}})^\circ _{\mathrm{contra}})} \\
	{\operatorname{Fun}(\mathcal{C}^{\mathrm{op}},\mathcal{S})}
	\arrow["{N(\operatorname{Un}_\varepsilon)}", from=1-1, to=1-2]
	\arrow["\simeq"', from=1-1, to=1-2]
	\arrow["\simeq"', from=1-1, to=2-1]
	\arrow[dashed, from=2-1, to=1-2]
\end{tikzcd}\]commutative, where $\pr{\SS_{/\cal C}}_{{\rm contra}}$ denotes the
contravariant model structure \cite[$\S$2.1.4]{HTT} and $\Un_{\varepsilon}$
denotes the unstraightening functor \cite[$\S$2.2.1]{HTT} with respect
to the counit map $\varepsilon:\fr C[\cal C]\to\cal C^{\Delta}$,
and the left vertical arrow is the categorical equivalence of \cite[Proposition 4.2.4.4]{HTT}.

Suppose now that $\cal C$ is a Kan complex. Then the contravariant
model structure on $\SS_{/\cal C}$ coincides with the Kan--Quillen
model structure, so there is a categorical equivalence $N\pr{\pr{\SS_{/\cal C}}_{{\rm contra}}^{\circ}}\xrightarrow{\simeq}\cal S_{/\cal C}$
\cite[\href{https://kerodon.net/tag/01ZT}{Tag 01ZT}]{kerodon}. We
will refer to the composite
\[
\Fun\pr{\cal C^{\op},\cal S}\to N\pr{\pr{\SS_{/\cal C}}_{{\rm contra}}^{\circ}}\xrightarrow{\simeq}\cal S_{/\cal C}
\]
also as the \textbf{unstraightening functor}.
\end{defn}

\begin{rem}
\label{rem:tangent_classifier}The functor $\tau:\Mfld_{n}\to N\pr{\pr{\SS_{/B\Top\pr n}}^{\circ}}$
is naturally equivalent to the composite
\[
\Mfld_{n}\xrightarrow{y}\Fun\pr{\Mfld_{n}^{\op},\cal S}\xrightarrow{i^{*}}\Fun\pr{B\Top\pr n^{\op},\cal S}\xrightarrow{\simeq}N\pr{\pr{\SS_{/B\Top\pr n}}^{\circ}},
\]
where $y$ denotes the Yoneda embedding, the second functor is the
restriction along the inclusion $i:B\Top\pr n\hookrightarrow\Mfld_{n}$,
and the last equivalence is the unstraightening equivalence. Indeed,
replacing $\Mfld_{n}$ by its full subcategory spanned by the manifolds
embedded in some Euclidean space, we may assume that $\Mfld_{n}$
is small. The naturality of unstraightening \cite[Appendix A]{GHN17}
implies that the diagram % https://q.uiver.app/#q=WzAsNCxbMCwwLCJcXG9wZXJhdG9ybmFtZXtGdW59KFxcbWF0aGNhbHtNfVxcbWF0aHNme2ZsZH1fbl57XFxtYXRocm17b3B9fSxcXG1hdGhjYWx7U30pIl0sWzEsMCwiTigoXFxtYXRoc2Z7c1NldH1fey9cXG1hdGhjYWx7TX1cXG1hdGhzZntmbGR9X259KV5cXGNpcmNfe1xcbWF0aHJte2NvbnRyYX19KSJdLFswLDEsIlxcb3BlcmF0b3JuYW1le0Z1bn0oQlxcbWF0aHJte1RvcH0obilee1xcbWF0aHJte29wfX0sXFxtYXRoY2Fse1N9KSJdLFsxLDEsIk4oKFxcbWF0aHNme3NTZXR9X3svQlxcbWF0aHJte1RvcH0obil9KV5cXGNpcmMpIl0sWzAsMSwiXFxzaW1lcSJdLFswLDJdLFsyLDMsIlxcc2ltZXEiLDJdLFsxLDNdXQ==
\[\begin{tikzcd}
	{\operatorname{Fun}(\mathcal{M}\mathsf{fld}_n^{\mathrm{op}},\mathcal{S})} & {N((\mathsf{sSet}_{/\mathcal{M}\mathsf{fld}_n})^\circ_{\mathrm{contra}})} \\
	{\operatorname{Fun}(B\mathrm{Top}(n)^{\mathrm{op}},\mathcal{S})} & {N((\mathsf{sSet}_{/B\mathrm{Top}(n)})^\circ)}
	\arrow["\simeq", from=1-1, to=1-2]
	\arrow[from=1-1, to=2-1]
	\arrow[from=1-2, to=2-2]
	\arrow["\simeq"', from=2-1, to=2-2]
\end{tikzcd}\]commutes up to natural equivalence, where the horizontal arrows are
the unstraightening equivalences. Thus, it suffices to show that $\tau$
is naturally equivalent to the composite
\[
\Mfld_{n}\xrightarrow{y}\Fun\pr{\Mfld_{n}^{\op},\cal S}\xrightarrow{\simeq}N\pr{\pr{\SS_{/\Mfld_{n}}}_{{\rm contra}}^{\circ}}\to N\pr{\pr{\SS_{/B\Top\pr n}}^{\circ}},
\]
which follows from the definitions. 
\end{rem}

\subsection{\label{subsec:univ_colim}Universal Colimit Diagrams}

Let $\cal C$ be an ordinary category with pullbacks and small colimits.
In ordinary category theory, we say that \textbf{colimits in $\cal C$
are universal} if for each morphism $f:X\to Y$ in $\cal C$, the
pullback functor $f^{*}:\cal C_{/Y}\to\cal C_{/X}$ preserves small
colimits. In some cases, colimits in $\cal C$ may not be universal,
but some particular colimit diagram in $\cal C_{/Y}$ is preserved
by the pullback functor $f^{*}$. The goal of this subsection is to
record some basic facts on such colimit diagrams in the setting of
$\infty$-categories.
\begin{defn}
Let $\cal C$ be an $\infty$-category with pullbacks, let $K$ be
a simplicial set, and let $p:K^{\rcone}\to\cal C$ be a diagram. We
say that $p$ is a \textbf{universal colimit diagram} if for each
cartesian natural transformation $\alpha:K^{\rcone}\times\Delta^{1}\to\cal C$
such that $\alpha\vert K^{\rcone}\times\{1\}=p$, the diagram $\alpha\vert K^{\rcone}\times\{0\}$
is a colimit diagram.
\end{defn}

\begin{defn}
Let $\cal C$ be an $\infty$-category with pullbacks. Given a morphism
$f:X\to Y$ in $\cal C$, we let $f^{*}:\cal C_{/Y}\to\cal C_{/X}$
denote the right adjoint of the composite
\[
\cal C_{/X}\xrightarrow[\phi]{\simeq}\cal C_{/f}\to\cal C_{/Y},
\]
where the functor $\phi$ is any section of the trivial fibration
$\cal C_{/f}\xrightarrow{\simeq}\cal C_{/X}$.
\end{defn}

\begin{rem}
\cite[Lemma 6.1.3.3]{HA} Let $\cal C$ be an $\infty$-category with
pullbacks, and let $p:K^{\rcone}\to\cal C$ be a diagram. The following
conditions are equivalent:
\begin{enumerate}
\item The diagram $p$ is a universal colimit diagram.
\item For every morphism $f:X\to Y$ in $\cal C$ and for every diagram
$p':K^{\rcone}\to\cal C_{/Y}$ lifting $p$, the composite
\[
K^{\rcone}\xrightarrow{p'}\cal C_{/Y}\xrightarrow{f^{*}}\cal C_{/X}
\]
is a colimit diagram.
\end{enumerate}
\iffalse

Indeed, condition (2) is equivalent to the condition that, for each
morphism $f:X\to Y$ in $\cal C$ and for every diagram $p':K^{\rcone}\to\cal C^{/Y}$
lifting $p$, the composite
\[
K^{\rcone}\xrightarrow{p'}\cal C^{/Y}\xrightarrow{f^{*}}\cal C^{/X}
\]
is a colimit diagram. Unwinding the definitions, this is equivalent
to the condition that for each cartesian natural transformation $\alpha:\pr{K^{\rcone}\diamond\Delta^{0}}\times\Delta^{1}\to\cal C$
such that $\alpha\vert K^{\rcone}\times\{0\}=p$ and $\alpha\vert\Delta^{0}\times\Delta^{1}=f$,
the restriction $\alpha\vert K^{\rcone}\times\{1\}$ is a colimit
diagram. The claim then follows from the observation that $K^{\rcone}$
is a retract of $K^{\rcone}\diamond\Delta^{0}$.

\fi
\end{rem}

\begin{defn}
Let $f:\cal C\to\cal D$ be a functor of $\infty$-categories, and
let $\cal C'\subset\cal C$ be a full subcategory. Suppose that $\cal D$
has pullbacks. We say that $f$ is a \textbf{universal left Kan extension}
of $f\vert\cal C'$ if for each object $X\in\cal C$, the diagram
\[
\pr{\cal C'_{/X}}^{\rcone}\to\cal C\xrightarrow{f}\cal D
\]
is a universal colimit diagram.
\end{defn}

\begin{prop}
\label{prop:uLan}Let $f:\cal C\to\cal D$ be a functor of $\infty$-categories,
and let $\cal C'\subset\cal C$ be a full subcategory. If $f$ is
a universal left Kan extension of $f\vert\cal C'$, then for every
object $X\in\cal C$, the composite
\[
f_{X}:\cal C_{/X}\to\cal C\xrightarrow{f}\cal D
\]
is a universal left Kan extension of $f_{X}\vert\cal C'_{/X}$. 
\end{prop}

\begin{proof}
It suffices to show that, for each object $p:C\to X$ in $\cal C_{/X}$,
the functor
\[
\cal C'_{/X}\times_{\cal C_{/X}}\cal C_{/p}\to\cal C'\times_{\cal C}\cal C_{/C}
\]
is final. But this is a trivial fibration, being a pullback of the
trivial fibration $\cal C_{/p}\to\cal C_{/C}$.
\end{proof}

\subsection{\label{subsec:main}Main Result}

We now prove the main theorem of this paper (Theorem \ref{thm:intro1}).
Let us recall the statement of the theorem once again:
\begin{thm}
\label{thm:main}The diagram% https://q.uiver.app/#q=WzAsMyxbMiwwLCJcXG1hdGhjYWx7T31cXG1hdGhzZntwfV9cXGluZnR5Il0sWzAsMCwiXFxvcGVyYXRvcm5hbWV7RnVufShCXFxtYXRocm17VG9wfShuKV57XFxtYXRocm17b3B9fSxcXG1hdGhjYWx7U30pIl0sWzEsMSwiXFxtYXRoY2Fse1N9X3svQlxcbWF0aHJte1RvcH0obil9Il0sWzIsMCwiXFxtYXRoYmJ7RX1fXFxidWxsZXReXFxvdGltZXMgIiwyXSxbMSwwLCItXFxvdGltZXNfe1xcbWF0aHJte1RvcH0obil9IFxcbWF0aGJie0V9X3tcXG1hdGhiYntSfV5ufV5cXG90aW1lcyJdLFsxLDIsIlxcc2ltZXEiLDJdXQ==
\[\begin{tikzcd}
	{\operatorname{Fun}(B\mathrm{Top}(n)^{\mathrm{op}},\mathcal{S})} && {\mathcal{O}\mathsf{p}_\infty} \\
	& {\mathcal{S}_{/B\mathrm{Top}(n)}}
	\arrow["{-\otimes_{\mathrm{Top}(n)} \mathbb{E}_{\mathbb{R}^n}^\otimes}", from=1-1, to=1-3]
	\arrow["\simeq"', from=1-1, to=2-2]
	\arrow["{\mathbb{E}_\bullet^\otimes }"', from=2-2, to=1-3]
\end{tikzcd}\]of $\infty$-categories commutes up to natural equivalences, where
the left slanted arrow is the unstraightening functor (Definition
\ref{def:Un}).
\end{thm}

The proof relies on the following lemma, which we will prove in Section
\ref{sec:lemma}.
\begin{lem}
\label{lem:key}Let $\cal D\subset\cal S$ denote the full subcategory
spanned by the contractible Kan complexes. The functor
\[
\pr -^{\amalg}:\cal S\to\Op_{\infty}
\]
is a universal left Kan extension of $\pr -^{\amalg}\vert\cal D$.
\end{lem}

\begin{proof}
[Proof of Theorem \ref{thm:main}, assuming Lemma \ref{lem:key}]Let
$y:B\Top\pr n\to\Fun\pr{B\Top\pr n^{\op},\cal S}$ denote the Yoneda
embedding. By construction, the diagram of $\infty$-categories % https://q.uiver.app/#q=WzAsNCxbMiwxLCJcXG1hdGhjYWx7T31cXG1hdGhzZntwfV9cXGluZnR5Il0sWzAsMSwiXFxvcGVyYXRvcm5hbWV7RnVufShCXFxtYXRocm17VG9wfShuKV57XFxtYXRocm17b3B9fSxcXG1hdGhjYWx7U30pIl0sWzEsMSwiXFxtYXRoY2Fse1N9X3svQlxcbWF0aHJte1RvcH0obil9Il0sWzAsMCwiQlxcbWF0aHJte1RvcH0obikiXSxbMiwwLCJcXG1hdGhiYntFfV9cXGJ1bGxldF5cXG90aW1lcyAiLDJdLFsxLDIsIlxcc2ltZXEiLDJdLFszLDEsInkiLDJdLFszLDAsIlxcbWF0aGJie0V9Xlxcb3RpbWVzIF97XFxtYXRoYmJ7Un1ebn0iXV0=
\[\begin{tikzcd}
	{B\mathrm{Top}(n)} \\
	{\operatorname{Fun}(B\mathrm{Top}(n)^{\mathrm{op}},\mathcal{S})} & {\mathcal{S}_{/B\mathrm{Top}(n)}} & {\mathcal{O}\mathsf{p}_\infty}
	\arrow["y"', from=1-1, to=2-1]
	\arrow["{\mathbb{E}^\otimes _{\mathbb{R}^n}}", from=1-1, to=2-3]
	\arrow["\simeq"', from=2-1, to=2-2]
	\arrow["{\mathbb{E}_\bullet^\otimes }"', from=2-2, to=2-3]
\end{tikzcd}\]commutes up to natural equivalence. Therefore, it will suffice to
show that the functor $\bb E_{\bullet}^{\t}:\cal S_{/B\Top\pr n}\to\Op_{\infty}$
is a left Kan extension of its restriction to $B\Top\pr n$. 

Let $\cal D\subset\cal S$ denote the full subcategory spanned by
the contractible Kan complexes. Then the essential image of the composite
$B\Top\pr n\xrightarrow{y}\Fun\pr{B\Top\pr n^{\op},\cal S}\xrightarrow{\simeq}\cal S_{/B\Top\pr n}$
is equal to $\cal D_{/B\Top\pr n}$. Since the Yoneda embedding is
fully faithful, this implies that the functor $B\Top\pr n\to\cal S_{/B\Top\pr n}$
restricts to a categorical equivalence $B\Top\pr n\xrightarrow{\simeq}\cal D_{/B\Top\pr n}$.
Therefore, we are reduced to showing that the functor $\bb E_{\bullet}^{\t}:\cal S_{/B\Top\pr n}\to\Op_{\infty}$
is a left Kan extension of $\bb E_{\bullet}^{\t}\vert\cal D_{/B\Top\pr n}$.
Let $\iota:\bb E_{B\Top\pr n}^{\t}\to B\Top\pr n^{\amalg}$ denote
the inclusion. By Remark \ref{rem:amalg}, the functor $\bb E_{\bullet}^{\t}$
is naturally equivalent to the composite
\[
\cal S_{/B\Top\pr n}\xrightarrow{\pr -^{\amalg}}\pr{\Op_{\infty}}_{/B\Top\pr n^{\amalg}}\xrightarrow{\iota^{*}}\pr{\Op_{\infty}}_{/\bb E_{B\Top\pr n}^{\t}}\xrightarrow{U}\Op_{\infty},
\]
where $U$ denotes the forgetful functor. According to Proposition
\ref{prop:uLan} and Lemma \ref{lem:key}, the composite 
\[
\cal S_{/B\Top\pr n}\xrightarrow{\pr -^{\amalg}}\pr{\Op_{\infty}}_{/B\Top\pr n}\to\Op_{\infty}
\]
is a universal left Kan extension of $\cal D_{/B\Top\pr n}$, where
$\cal D\subset\cal S$ denotes the full subcategory spanned by the
contractible Kan complexes. It follows that the functor $\iota^{*}\circ\pr -^{\amalg}:\cal S_{/B\Top\pr n}\to\pr{\Op_{\infty}}_{/\bb E_{B\Top\pr n}^{\t}}$
is a left Kan extension of its restriction to $\cal D_{/B\Top\pr n}$.
Since $U$ preserves small colimits, we are done.
\end{proof}

\subsection{\label{subsec:glue}Matsuoka's Gluing Theorem}

As an application of Theorem \ref{thm:intro1}, we give an alternative
proof of Matsuoka's gluing theorem on locally constant factorization
algebras. We say that a functor $F:\Mfld_{n}\to\cal C$ of $\infty$-categories
is a \textbf{cosheaf }if for each $n$-manifold $M$ and each open
cover $\cal U$ of $M$ which is downward-closed (i.e., if $U\in\cal U$,
then every open set of $U$ belongs to $\cal U$), the map $\colim_{U\in N\pr{\cal U}}FU\to FM$
is an equivalence.\footnote{Cosheaves on $\Mfld_{n}$ with values in $\infty$-categories with
small colimits admit various characterizations, such as being left
Kan extended from their restriction to $B\Top\pr n$, or being ($1$-)excisive
and exhaustive. For a proof, see \cite[Theorem 5.3]{A24b}.} 
\begin{cor}
\cite{Matsu17}\label{cor:FA_glue} The functor
\[
\bb E_{\bullet}^{\otimes}:\Mfld_{n}\to\Op_{\infty}
\]
(of Definition \ref{def:E_M}) is a cosheaf.
\end{cor}

\begin{proof}
By Theorem \ref{thm:intro1} and \cite[Theorem 5.1.5.6]{HTT}, the
functor $\bb E_{\bullet}^{\t}:N\pr{\pr{\SS_{/B\Top\pr n}}^{\circ}}\to\Op_{\infty}$
preserves small colimits. Therefore, it suffices to show that the
tangent classifier functor $\tau:\Mfld_{n}\to N\pr{\pr{\SS_{/B\Top\pr n}}^{\circ}}$
is a cosheaf. As observed in Remark \ref{rem:tangent_classifier},
the composite $\Mfld_{n}\xrightarrow{\tau}N\pr{\pr{\SS_{/B\Top\pr n}}^{\circ}}\xrightarrow{\simeq}\cal S_{/B\Top\pr n}$
is naturally equivalent to the composite
\[
\Mfld_{n}\xrightarrow{y}\Fun\pr{\Mfld_{n}^{\op},\cal S}\xrightarrow{i^{*}}\Fun\pr{B\Top\pr n^{\op},\cal S}\simeq\cal S_{/B\Top\pr n},
\]
where $y$ denotes the Yoneda embedding, the second functor is the
restriction along the inclusion $i:B\Top\pr n\hookrightarrow\Mfld_{n}$,
and the last equivalence is the unstraightening equivalence. It will
therefore suffice to show that the composite $i^{*}\circ y$ is a
cosheaf (Remark \ref{rem:tangent_classifier}). Since colimits in
functor categories can be computed pointwise \cite[\href{https://kerodon.net/tag/02XK}{Tag 02XK}]{kerodon},
it suffices to show that the composite
\[
\Mfld_{n}\xrightarrow{y}\Fun\pr{\Mfld_{n}^{\op},\cal S}\xrightarrow{i^{*}}\Fun\pr{B\Top\pr n^{\op},\cal S}\xrightarrow{\opn{ev}_{\bb R^{n}}}\cal S
\]
is a cosheaf. In other words, we are reduced to showing that the functor
$\Sing\Emb\pr{\bb R^{n},-}:\Mfld_{n}\to\cal S$ is a cosheaf.

Let $M$ be an $n$-manifold and let $\cal U$ be an open cover of
$M$ which is downward-closed. We wish to show that the map
\[
\colim_{U\in N\pr{\cal U}}\Sing\Emb\pr{\bb R^{n},U}\to\Sing\Emb\pr{\bb R^{n},M}
\]
is an equivalence. According to \cite[Proposition 2.19]{A24b}, for
each $U\in\cal U$, the evaluation at the origin determines a pullback
square % https://q.uiver.app/#q=WzAsNCxbMCwwLCJcXG9wZXJhdG9ybmFtZXtTaW5nfVxcb3BlcmF0b3JuYW1le0VtYn0oXFxtYXRoYmJ7Un1ebixVKSJdLFsxLDAsIlxcb3BlcmF0b3JuYW1le1Npbmd9XFxvcGVyYXRvcm5hbWV7RW1ifShcXG1hdGhiYntSfV5uLE0pIl0sWzEsMSwiXFxvcGVyYXRvcm5hbWV7U2luZ31NIl0sWzAsMSwiXFxvcGVyYXRvcm5hbWV7U2luZ31VIl0sWzAsMV0sWzEsMl0sWzAsM10sWzMsMl1d
\[\begin{tikzcd}
	{\operatorname{Sing}\operatorname{Emb}(\mathbb{R}^n,U)} & {\operatorname{Sing}\operatorname{Emb}(\mathbb{R}^n,M)} \\
	{\operatorname{Sing}U} & {\operatorname{Sing}M}
	\arrow[from=1-1, to=1-2]
	\arrow[from=1-2, to=2-2]
	\arrow[from=1-1, to=2-1]
	\arrow[from=2-1, to=2-2]
\end{tikzcd}\]in $\cal S$. Since colimits in $\cal S$ are universal \cite[Lemma 6.1.3.14]{HTT},
we are reduced to showing that the map
\[
\colim_{U\in N\pr{\cal U}}\Sing U\to\Sing M
\]
is an equivalence. This follows from \cite[Theorem A.3.1]{HA}.
\end{proof}

\section{\label{sec:lemma}Proof of Lemma \ref{lem:key}}

The goal of this section is to prove Lemma \ref{lem:key}, which asserts
that the functor $\pr -^{\amalg}:\cal S\to\Op_{\infty}$ is a universal
left Kan extension of its restriction to the full subcategory of contractible
Kan complexes. 
\begin{warning}
(Casual readers may safely ignore this warning.) There are two conflicting
conventions of the homotopy coherent nerve functor:
\begin{enumerate}
\item The homotopy coherent nerve functor $N_{I}$, defined as in \cite{HTT}.
\item The homotopy coherent nerve functor $N_{II}$, defined as follows:
If $\cal C$ is a simplicial category, we set $N_{II}\pr{\cal C}=N_{I}\pr{\cal C^{c}}$,
where $\cal C^{c}$ denotes the simplicial category obtained from
$\cal C$ by replacing each hom-simplicial sets by its opposites.
(This is the convention adopted in \cite{kerodon}.)
\end{enumerate}
As we remarked in Section \ref{sec:intro}, this paper generally follows
\cite{HTT, HA} in its terminology and notation. This means that,
so far, we have adopted convention (1). However, in \cite{A24c},
which we will frequently refer to below, the author used convention
(2) (as it seemed more natural to do so\footnote{A rule of thumb is that when we want to consider straightening--unstraightening
of cartesian fibrations, we should use $N_{I}$, while we consider
that of cocartesian fibrations, we should use $N_{II}$.}). Because of this, we will henceforth switch to convention (2). Thus,
from now on, the $\infty$-categories such as $\cal S$, $\Cat_{\infty}$,
$\Op_{\infty}$ will be defined by applying the functor $N_{II}$
to the simplicial categories $\SS^{\circ}$, $\pr{\SS^{+}}^{\circ}$,
and $\Op_{\infty}^{\Delta}$. Note that this is allowed as far as
Lemma \ref{lem:key} is concerned, for the validity of the lemma does
not depend on the choice of the convention.
\end{warning}

\subsection{\label{subsec:recollections}Recollections}

In this subsection, we review some results on categorical patterns
\cite[Appendix B]{HA} that are proved in \cite{A24c}. 
\begin{defn}
\cite[Definition B.0.19]{HA} Let $S$ be a simplicial set. A \textbf{categorical
pattern} $\frak P=\pr{M_{S},T,\{p_{\alpha}\}_{\alpha\in A}}$ on $S$
consists of a set $M_{S}$ of edges of $S$ containing all degenerate
edges, a set $T$ of $2$-simplices of $S$ containing all degenerate
$2$-simplices, and a set $\{p_{\alpha}:K_{\alpha}^{\lcone}\to S\}_{\alpha\in A}$
of diagrams of $S$ such that $p_{\alpha}$ carries all edges and
$2$-simplices of $K_{\alpha}^{\lcone}$ into $M_{S}$ and $T$, respectively.

In the case where $T$ contains all $2$-simplices of $S$, we will
omit $T$ from the notation and write $\frak P=\pr{M_{S},\{p_{\alpha}\}_{\alpha\in A}}$.
(All categorical patterns we consider in this paper are of this form.)
If further $S$ is an $\infty$-category and $M_{S}$ contains all
equivalences of $S$, we say that $\frak P$ is a \textbf{commutative
categorical pattern} \cite[Definition 2.14]{A24c}.
\end{defn}

\begin{defn}
\cite[Definition B.0.19]{HA}, \cite[Remark 2.5]{A24c}\label{def:P-fib}
Let $\frak P=\pr{M_{S},\{p_{\alpha}\}_{\alpha\in A}}$ be a categorical
pattern on a simplicial set $S$. A map $p:X\to S$ of simplicial
sets is said to be \textbf{$\frak P$-fibered} if it satisfies the
following conditions:
\begin{enumerate}
\item The map $p$ is an inner fibration.
\item For every edge $s\to s'$ in $M_{S}$ and every vertex $x\in X$ lying
over $s$, there is a $p$-cocartesian edge $x\to x'$ lying over
$M_{S}$.
\item Each $p_{\alpha}$ lifts to a map $\widetilde{p}_{\alpha}:K_{\alpha}^{\lcone}\to X$
which carries all edges to $p$-cocartesian edges. Moreover, any such
lift is a $p$-limit diagram.
\end{enumerate}
If $p$ satisfies these conditions, we will write $X_{\natural}$
for the marked simplicial set obtained from $X$ by marking the $p$-cocartesian
morphisms whose images in $S$ belong to $M_{S}$. 
\end{defn}

The totality of $\frak P$-fibered objects can be organized into an
$\infty$-category, because of the following theorem:
\begin{thm}
\cite[Theorem B.0.20]{HA} Let $\frak P=\pr{M_{S},\{p_{\alpha}\}_{\alpha\in A}}$
be a categorical pattern on a simplicial set $S$. There is a combinatorial
simplicial model structure on $\SS_{/\pr{S,M_{S}}}^{+}$, denoted
by $\SS_{/\frak P}^{+}$, whose cofibrations are the monomorphisms
and whose fibrant objects are the objects of the form $X_{\natural}$,
where $X\to S$ is $\frak P$-fibered.
\end{thm}

\begin{defn}
Let $\frak P=\pr{M_{S},\{p_{\alpha}\}_{\alpha\in A}}$ be a categorical
pattern on a simplicial set $S$. We write $\frak P\-\Fib$ for the
homotopy coherent nerve of the full simplicial subcategory of $\SS_{/\frak P}^{+}$
spanned by the fibrant-cofibrant objects.
\end{defn}

\begin{example}
\label{exa:oo-operad}For each $n\geq0$, let $\rho_{n}:\pr{\{1,\dots,n\}}^{\lcone}\to N\pr{\Fin_{\ast}}$
denote the functor which classifies the $n$ inert morphisms $\inp n\to\inp 1$.
(When $n=0$, the diagram $\rho_{n}$ classifies the object $\inp 0\in N\pr{\Fin_{\ast}}$.)
A functor $\cal E\to N\pr{\Fin_{\ast}}$ is fibered over the categorical
pattern $\frak{Op}=\pr{\{\text{inert maps}\},\{\rho_{n}\}_{n\geq0}}$
if and only if it is an $\infty$-operad. By definition, we have $\frak{Op}\-\Fib=\Op_{\infty}$.
\end{example}

Just like the ordinary straightening--unstraightening, functors with
values in $\frak P\-\Fib$ can equivalently be specified by a fibrational
structure, which we now review.
\begin{defn}
\label{def:P-bundle}\cite[Definition 3.1]{A24c} Let $\frak P=\pr{M_{\cal D},\{p_{\alpha}\}_{\alpha\in A}}$
be a commutative categorical pattern on an $\infty$-category $\cal D$,
and let $S$ be a simplicial set. A \textbf{$\frak P$-bundle }(\textbf{over
}$S$) is a commutative diagram % https://q.uiver.app/#q=WzAsMyxbMCwwLCJYIl0sWzIsMCwiU1xcdGltZXMgXFxtYXRoY2Fse0R9Il0sWzEsMSwiUyJdLFswLDEsInAiXSxbMSwyLCJcXG9wZXJhdG9ybmFtZXtwcn0iXSxbMCwyLCJxIiwyXV0=
\[\begin{tikzcd}
	X && {S\times \mathcal{D}} \\
	& S
	\arrow["p", from=1-1, to=1-3]
	\arrow["{\operatorname{pr}}", from=1-3, to=2-2]
	\arrow["q"', from=1-1, to=2-2]
\end{tikzcd}\]of simplicial sets which satisfies the following conditions:

\begin{enumerate}[label=(\alph*)]

\item The map $q:X\to S$ is a cocartesian fibration.

\item The map $p$ lifts to a fibration of fibrant objects of $\SS_{/S}^{+}$
with respect to the cocartesian model structure.

\item For each vertex $v\in S$, the map $X_{v}=X\times_{S}\{v\}\to\cal D$
is $\frak P$-fibered.

\item For each edge $f:v\to v'$ in $S$, the induced functor $f_{!}:X_{v}\to X_{v'}$
is a morphism of $\frak P$-fibered objects.

\end{enumerate}

We will often say that the map $p$ (or $X$) is a $\frak P$-bundle
over $S$. Given a $\frak P$-bundle $p:X\to S\times\cal D$, we will
write $X_{\natural}$ for the marked simplicial set obtained from
$X$ by marking the $p$-cocartesian edges whose images in $\cal D$
belong to $M_{\cal D}$. This does not conflict with the notation
in Definition \ref{def:P-fib}, because of the following reason: Let
$S\times\frak P$ denote the categorical pattern on $S\times\cal D$
given by
\[
S\times\frak P=\pr{S_{1}\times M_{\cal D},\{\{v\}\times p_{\alpha}\}_{v\in S_{0},\,\alpha\in A}}.
\]
We can show that the fibrant-cofibrant objects of $\SS_{/S\times\frak P}^{+}$
are precisely the objects of the form $\pr{X,M_{X}}\to S^{\sharp}\times\pr{\cal D,M_{D}}$,
where $X$ is a $\frak P$-bundle and $M_{X}$ is the set of $p$-cocartesian
edges whose images in $\cal D$ belong to $M_{\cal D}$ \cite[Proposition 3.5]{A24c}. 

We will write $\frak P\-\Bund\pr S$ for the homotopy coherent nerve
of the full simplicial subcategory of $\SS_{/S\times\frak P}^{+}$
spanned by the fibrant--cofibrant objects. 
\end{defn}

The following is the main result of \cite{A24c}:
\begin{thm}
\cite[Corollary 5.10]{A24c}\label{thm:A24cmain} Let $\frak P$ be
a commutative categorical pattern, and let $S$ be a small simplicial
set. There is a categorical equivalence
\[
\frak P\-\Bund\pr S\simeq\Fun\pr{S,\frak P\-\Fib}
\]
which lifts the ordinary straightening--unstraightening equivalence.
\end{thm}

\begin{rem}
\label{rem:A24cmain}Let $\sf{CCP}$ denote the category whose objects
are the pairs $\pr{\cal D,\frak P}$, where $\cal D$ is an $\infty$-category
and $\frak P$ is a commutative categorical pattern on $\cal D$,
and whose morphisms $\pr{\cal D,\frak P}\to\pr{\cal D',\frak P'}$
are functors $\cal D\to\cal D'$ that carry each edge and diagram
in $\frak P$ into those of $\frak P$. Then the assignments $\pr{S,\pr{\cal D,\frak P}}\mapsto\frak P\-\Bund\pr S$
and $\pr{S,\pr{\cal D,\frak P}}\mapsto\Fun\pr{S,\frak P\-\Fib}$ determines
a functor $N\pr{\SS^{\op}\times\sf{CCP}^{\op}}\to\hat{\Cat}_{\infty}$,
and the equivalence of Theorem \ref{thm:A24cmain} can be promoted
to a natural equivalence of these functors. \iffalse i.e., the equivalence of the theorem is a component of the said natural equivalence \fi
This follows from the proof of \cite[Corollary 5.10]{A24c} (and arguing
as in the proof of the naturality of the ordinary straightening--unstraightening
\cite[Appendix A]{GHN17}).

In the situation of Theorem \ref{thm:A24cmain}, we say that a $\frak P$-bundle
$p:X\to S\times\cal D$ is \textbf{classified} by a functor $f:S\to\frak P\-\Fib$
if the equivalence of Theorem \ref{thm:A24cmain} carries $p$ to
an object equivalent to $f$. The naturality property discussed in
the previous paragraph implies that classifying maps are compatible
with pullback in the following two senses:
\begin{itemize}
\item If $S'\to S$ is a map of simplicial sets, then the $\frak P$-bundle
$X\times_{S}S'\to S'\times\cal D$ is classified by the composite
$S'\to S\xrightarrow{f}\frak P\-\Fib$. 
\item If $\frak P'$ is a commutative categorical pattern on an $\infty$-category
$\cal D'$ and $\cal D'\to\cal D$ is a functor which carries each
edge and each diagram of $\frak P'$ into those of $\frak P$, then
the $p$-bundle $X\times_{\cal D}\cal D'\to S\times\cal D'$ is classified
by the composite $S\xrightarrow{f}\frak P\-\Fib\to\frak P'\-\Fib$.
\end{itemize}
\end{rem}

\begin{example}
\cite[Proposition 5.12]{A24c}\label{exa:relnerve} Let $\frak P=\pr{M_{\cal D},\{p_{\alpha}\}_{\alpha}}$
be a commutative categorical pattern on an $\infty$-category $\cal D$,
let $\cal C$ be an ordinary category, and let $F:\cal C\to\SS_{/\frak P}^{+}$
be a functor which carries every object to a fibrant object. Then
the nerve of $F$ classifies the $\frak P$-bundle
\[
\int_{\cal C}F_{\flat}\to\int_{\cal C}\delta\pr{\cal D}\cong N\pr{\cal C}\times\cal D,
\]
where $F_{\flat}$ denotes the composite $\cal C\to\SS_{/\frak P}^{+}\xrightarrow{\text{forget}}\SS$,
$\int_{\cal C}F_{\flat}$ denotes the relative nerve of $F_{\flat}$
\cite[$\S$3.2.5]{HTT}, and $\delta\pr{\cal D}:\cal C\to\SS$ denotes
the constant functor at $\cal D$.
\end{example}

As in ordinary straightening--unstraightening, we can use Theorem
\ref{thm:A24cmain} to give a criterion for a diagram in $\frak P\-\Fib$
to be a colimit diagram. The criterion relies on the following preliminary
construction.
\begin{defn}
\cite[Definition 6.7]{A24c}\label{def:refraction} Let $\frak P=\pr{M_{\cal D},\{p_{\alpha}\}_{\alpha\in A}}$
be a commutative categorical pattern on an $\infty$-category $\cal D$,
let $K$ be a simplicial set, and let $p':X'\to K^{\rcone}\times\cal D$
be a $\frak P$-bundle over $K^{\rcone}$. Set $X=X'\times_{K^{\rcone}}K$.
Regarding $X'$ and $X$ as $\frak P$-bundles over $K^{\rcone}$
and $K$, respectively, we define objects $X'_{\natural}\in\SS_{/K^{\rcone}\times\frak P}^{+}$
and $X_{\natural}\in\SS_{/K\times\frak P}^{+}$ as in Definition \ref{def:P-bundle}.
A map $\opn{Rf}:X_{\natural}\to X'_{\natural}\times_{\pr{K^{\rcone}}^{\sharp}}\{\infty\}^{\sharp}$
of $\SS_{/\frak P}^{+}$ is called a \textbf{refraction map} if there
is a morphism $H:\pr{\Delta^{1}}^{\sharp}\times X_{\natural}\to X'_{\natural}$
in $\SS_{/K^{\rcone}\times\frak P}^{+}$ satisfying the following
conditions:
\begin{enumerate}
\item The diagram% https://q.uiver.app/#q=WzAsNSxbMCwwLCJcXHswXFx9Xlxcc2hhcnBcXHRpbWVzIFhfXFxuYXR1cmFsIl0sWzAsMSwiKFxcRGVsdGFeMSleXFxzaGFycCBcXHRpbWVzIFhfXFxuYXR1cmFsIl0sWzIsMCwiWCdfe1xcbmF0dXJhbH0iXSxbMiwxLCIoS15cXHRyaWFuZ2xlcmlnaHQpXlxcc2hhcnBcXHRpbWVzIFxcb3ZlcmxpbmV7XFxtYXRoY2Fse0R9fSJdLFsxLDEsIihcXERlbHRhXjEpXntcXHNoYXJwfVxcdGltZXMgKEspXlxcc2hhcnBcXHRpbWVzIFxcb3ZlcmxpbmV7XFxtYXRoY2Fse0R9fSJdLFswLDEsIiIsMCx7InN0eWxlIjp7InRhaWwiOnsibmFtZSI6Imhvb2siLCJzaWRlIjoidG9wIn19fV0sWzEsNCwiXFxvcGVyYXRvcm5hbWV7aWR9XFx0aW1lcyBwIiwyXSxbNCwzLCJoXFx0aW1lcyBcXG9wZXJhdG9ybmFtZXtpZH0iLDJdLFsyLDMsInAiXSxbMCwyXSxbMSwyLCJIIiwxXV0=
\[\begin{tikzcd}
	{\{0\}^\sharp\times X_\natural} && {X'_{\natural}} \\
	{(\Delta^1)^\sharp \times X_\natural} & {(\Delta^1)^{\sharp}\times (K)^\sharp\times \overline{\mathcal{D}}} & {(K^\triangleright)^\sharp\times \overline{\mathcal{D}}}
	\arrow[hook, from=1-1, to=2-1]
	\arrow["{\operatorname{id}\times p}"', from=2-1, to=2-2]
	\arrow["{h\times \operatorname{id}}"', from=2-2, to=2-3]
	\arrow["p", from=1-3, to=2-3]
	\arrow[from=1-1, to=1-3]
	\arrow["H"{description}, from=2-1, to=1-3]
\end{tikzcd}\]is commutative, where $h:\Delta^{1}\times K\to K$ is the map determined
by the inclusion $K\times\{0\}\hookrightarrow K^{\rcone}$ and the
projection $K\times\{1\}\to\{\infty\}$ and $\overline{\cal D}=\pr{\cal D,M_{\cal D}}$.
\item The restriction $H\vert\{1\}^{\sharp}\times X_{\natural}$ is equal
to $\opn{Rf}$.
\end{enumerate}
Note that refraction maps exist and are unique up to homotopy in the
model category $\SS_{/\frak P}^{+}$.
\end{defn}

Here is the colimit criterion.
\begin{prop}
\cite[Proposition 6.8]{A24c}\label{prop:A24c.6.8} Let $\frak P$
be a commutative categorical pattern on an $\infty$-category $\cal D$,
let $K$ be a small simplicial set, let $f:K^{\rcone}\to\frak P\-\Fib$
be a diagram which classifies a $\frak P$-bundle $X'\to K^{\rcone}\times\cal D$.
Set $X=X'\times_{K^{\rcone}}K$. The following conditions are equivalent:
\begin{enumerate}
\item The diagram $f$ is a colimit diagram.
\item The refraction map $X_{\natural}\to X'_{\natural}\times_{\pr{K^{\rcone}}^{\sharp}}\{\infty\}$
is a $\frak P$-equivalence.
\end{enumerate}
\end{prop}

We conclude this section with a remark comparing bundles of $\infty$-operads
with Lurie's \textit{families of $\infty$-operads}.
\begin{rem}
\label{rem:family_vs_bundle}For the categorical pattern $\frak{Op}$
for $\infty$-operads (Example \ref{exa:oo-operad}), the notion of
$\frak{Op}$-bundles is closely related to that of families of $\infty$-operads
\cite[Definition 2.3.2.10]{HA}. More precisely, if $\cal C$ is an
$\infty$-category, then every $\frak{Op}$-bundle over $\cal C$
is a family of $\infty$-operads. This is immediate from the definitions
and \cite[Corollary 4.3.1.15]{HTT}.

We can also prove this by using model categories: To see this, it
will be convenient to introduce some notation. Given a marked simplicial
set $\overline{S}=\pr{S,M_{S}}$ and a commutative categorical pattern
$\frak P=\pr{M_{\cal D},\{p_{\alpha}\}_{\alpha\in A}}$ on an $\infty$-category
$\cal D$, let $\overline{S}\times\frak P$ denote the pair $\pr{M_{S}\times M_{\cal D},\{\{v\}\times p_{\alpha}:\{v\}\times K_{\alpha}^{\lcone}\to S\times\cal D\}_{v\in S_{0},\,\alpha\in A}}$.
Unwinding the definitions, the fibrant objects of $\SS_{/\cal C^{\natural}\times\frak{Op}}^{+}$
are precisely the $\cal C$-families of $\infty$-operads whose inert
morphisms are marked. Since the pullback functor $\SS_{/\cal C\times\frak{Op}}^{+}=\SS_{/\cal C^{\sharp}\times\frak{Op}}^{+}\to\SS_{/\cal C^{\natural}\times\frak{Op}}^{+}$
is right Quillen (\cite[Proposition B.2.9]{HA}, every $\frak{Op}$-bundle
over $\cal C$ is a $\cal C$-family of $\infty$-operads.

Note that the above argument shows that if $\cal C$ is a Kan complex
(so that $\cal C^{\natural}=\cal C^{\sharp}$), then $\cal C$-families
of $\infty$-operads and bundles of $\infty$-operads over $\cal C$
are exactly the same things.
\end{rem}

\subsection{\label{subsec:univ_weq}Universal Weak Equivalences}

Recall that our goal of this section is to show that a certain diagram
in $\Op_{\infty}$ is a universal colimit diagram. For this, we will
need a version of Proposition \ref{prop:A24c.6.8} for universal colimit
diagrams (Proposition \ref{prop:univ_colim_criterion}), which is
the subject of this subsection.

To state the main result of this subsection, we need a model-categorical
notion associated with universal colimit diagrams, called universal
weak equivalences.
\begin{defn}
Let $\bf A$ be a model category. A morphism $A\to B$ of $\bf A$
is called a \textbf{universal weak equivalence }if for each fibration
$X\to B$ in $\bf A$, the map $A\times_{B}X\to X$ is a weak equivalence. 

If $\frak P$ is a categorical pattern, we will refer to universal
weak equivalences of $\SS_{/\frak P}^{+}$ as \textbf{universal $\frak P$-equivalences}.
\end{defn}

\begin{example}
\label{exa:univ_weq}Let $\bf A$ be a model category.
\begin{enumerate}
\item Every weak equivalence of fibrant objects of $\bf A$ is a universal
weak equivalence. This follows from \cite[Lemma A.2.4.3]{HTT}.
\item Let $A\xrightarrow{f}B\xrightarrow{g}C$ be morphisms of $\bf A$.
Suppose that $f$ is a universal weak equivalence. Then $g$ is a
universal weak equivalence if and only if $gf$ is a universal weak
equivalence.
\item Suppose that $\bf A$ is a simplicial model category in which every
object is cofibrant. We say that a morphism $i:A\to B$ is a \textbf{right
deformation retract} if there is a retraction $r:B\to A$ of $i$
and a map $h:\Delta^{1}\otimes B\to B$ such that $h\vert\{0\}\otimes B=\id_{B}$
and $h\vert\{1\}\otimes B=ir$. Every right deformation retract of
$\bf A$ is a universal weak equivalence. This follows from \cite[Proposition 6.15]{A24c}.
\item Suppose that $\bf A$ is a simplicial model category in which every
object is cofibrant. Part (3) (and its dual) implies that, for each
$A\in\bf A$ and $i\in\{0,1\}$, the map $\{i\}\otimes A\to\Delta^{1}\otimes A$
is a universal weak equivalence. Hence, by part (2), universal weak
equivalences of $\bf A$ are stable under left homotopy.
\end{enumerate}
\end{example}

Here is the main result of this subsection.
\begin{prop}
\label{prop:univ_colim_criterion}Let $\frak P$ be a commutative
categorical pattern on an $\infty$-category $\cal D$, let $K$ be
a small simplicial set, and let $f:K^{\rcone}\to\frak P\-\Fib$ be
a diagram which classifies a $\frak P$-bundle $X'\to K^{\rcone}\times\cal D$.
Set $X_{\natural}=X'_{\natural}\times_{\pr{K^{\rcone}}^{\sharp}}K^{\sharp}$
and $\pr{X'_{\infty}}_{\natural}=X'_{\natural}\times_{\pr{K^{\rcone}}^{\sharp}}\{\infty\}^{\sharp}$.
The following conditions are equivalent:
\begin{enumerate}
\item The diagram $f$ is a universal colimit diagram.
\item The refraction map $X_{\natural}\to\pr{X'_{\infty}}_{\natural}$ is
a universal $\frak P$-equivalence of $\SS_{/\frak P}^{+}$.
\end{enumerate}
\end{prop}

\begin{rem}
In the situation of Proposition \ref{prop:univ_colim_criterion},
if some refraction map is a universal $\frak P$-equivalence, so is
any other map (by point (4) of Example \ref{exa:univ_weq}). This
justifies the usage of the definite article (``the'') in condition
(2).
\end{rem}

The proof of Proposition \ref{prop:univ_colim_criterion} relies on
a few preliminaries.
\begin{defn}
Let $K$ be a weakly contractible simplicial set, and let $\cal C$
be an $\infty$-category. We say that a diagram $f:K\to\cal C$ is\textbf{
essentially constant }if $f$ carries each morphism to an equivalence
in $\cal C$.
\end{defn}

\begin{rem}
\label{rem:ess_cst}Let $K$ be a weakly contractible simplicial set,
and let $\cal C$ be an $\infty$-category. Then:
\begin{enumerate}
\item The diagonal functor $\delta:\cal C\to\Fun\pr{K,\cal C}$ is fully
faithful.
\item The essential image of $\delta$ consists of the essentially constant
diagrams.
\end{enumerate}
For part (1), we may assume that $\cal C$ has colimits of shape $K$
(by embedding $\cal C$ into a larger $\infty$-category if necessary).
In this case, $\delta$ is a fully faithful right adjoint by \cite[Corollary 4.4.4.10]{HTT}.
Part (2) follows from the observation that every diagram $K\to\cal C^{\simeq}$
is equivalent to a constant diagram, because $K$ is weakly contractible.
\end{rem}

\begin{lem}
\label{lem:univ_colim}Let $\cal C$ be an $\infty$-category with
pullbacks, let $K$ be a simplicial set, and let $p:K^{\rcone}\to\cal C$
be a diagram. The following conditions are equivalent:
\begin{enumerate}
\item The diagram $p$ is a universal colimit diagram.
\item For every pullback diagram % https://q.uiver.app/#q=WzAsNCxbMCwwLCJwJyJdLFsxLDAsInAiXSxbMSwxLCJxIl0sWzAsMSwicSciXSxbMCwxXSxbMSwyLCJcXGFscGhhIl0sWzMsMl0sWzAsM11d
\[\begin{tikzcd}
	{p'} & p \\
	{q'} & q
	\arrow[from=1-1, to=1-2]
	\arrow[from=1-1, to=2-1]
	\arrow["\alpha", from=1-2, to=2-2]
	\arrow[from=2-1, to=2-2]
\end{tikzcd}\]in $\Fun\pr{K^{\rcone},\cal C}$, if $q$ and $q'$ are essentially
constant, then the diagram $p'$ is a colimit diagram.
\item There exists a morphism $\alpha:p\to q$ in $\Fun\pr{K^{\rcone},\cal C}$
satisfying the following conditions:
\begin{enumerate}
\item The map $\alpha_{\infty}:p\pr{\infty}\to q\pr{\infty}$ is an equivalence.
\item The diagram $q$ is essentially constant.
\item For every pullback diagram % https://q.uiver.app/#q=WzAsNCxbMCwwLCJwJyJdLFsxLDAsInAiXSxbMSwxLCJxIl0sWzAsMSwicSciXSxbMCwxXSxbMSwyLCJcXGFscGhhIl0sWzMsMl0sWzAsM11d
\[\begin{tikzcd}
	{p'} & p \\
	{q'} & q
	\arrow[from=1-1, to=1-2]
	\arrow[from=1-1, to=2-1]
	\arrow["\alpha", from=1-2, to=2-2]
	\arrow[from=2-1, to=2-2]
\end{tikzcd}\]in $\Fun\pr{K^{\rcone},\cal C}$, if $q'$ is essentially constant,
then $p'$ is a colimit diagram.
\end{enumerate}
\end{enumerate}
\end{lem}

\begin{proof}
We first prove that (1)$\implies$(2). Suppose that condition (1)
is satisfied. In the situation of (2), iterated applications of the
pasting law of pullbacks (\cite[Lemma 4.4.2.1]{HTT}) shows that the
natural transformation $p'\to p$ is cartesian. Hence $p'$ is a colimit
diagram, proving (1)$\implies$(2).

Next, we prove (2)$\implies$(1). Suppose that condition (2) is satisfied.
Let $\alpha:p'\to p$ be a cartesian natural transformation of diagrams
$K^{\rcone}\to\cal C$. We wish to show that $p'$ is a colimit diagram.
Pulling back $\alpha$ along the natural transformation $K^{\rcone}\times\Delta^{1}\to K^{\rcone}$
from the identity map to the constant map at the cone point, we obtain
a cartesian square % https://q.uiver.app/#q=WzAsNCxbMCwwLCJwJyJdLFsxLDAsInAiXSxbMSwxLCJcXGRlbHRhKHAoXFxpbmZ0eSkpIl0sWzAsMSwiXFxkZWx0YShwJyhcXGluZnR5KSkiXSxbMCwxXSxbMSwyXSxbMywyXSxbMCwzXV0=
\[\begin{tikzcd}
	{p'} & p \\
	{\delta(p'(\infty))} & {\delta(p(\infty))}
	\arrow[from=1-1, to=1-2]
	\arrow[from=1-1, to=2-1]
	\arrow[from=1-2, to=2-2]
	\arrow[from=2-1, to=2-2]
\end{tikzcd}\]in $\Fun\pr{K^{\rcone},\cal C}$, where $\delta:\cal C\to\Fun\pr{K^{\rcone},\cal C}$
denotes the diagonal functor. Condition (2) then tells us that $p'$
is a colimit diagram. Hence (2)$\implies$(1).

Finally, we prove (2)$\iff$(3). It is clear that (2)$\implies$(3).
For the converse, it suffices to show that for every morphism $\beta:p\to r$
in $\Fun\pr{K^{\rcone},\cal C}$ such that $r$ is essentially constant,
there is a diagram $\Delta^{2}\to\Fun\pr{K^{\rcone},\cal C}$ whose
boundary is depicted as % https://q.uiver.app/#q=WzAsMyxbMSwwLCJwIl0sWzAsMSwicSJdLFsyLDEsInIiXSxbMCwxLCJcXGFscGhhIiwyXSxbMCwyLCJcXGJldGEgIl0sWzEsMiwiXFxnYW1tYSIsMl1d
\[\begin{tikzcd}
	& p \\
	q && r.
	\arrow["\alpha"', from=1-2, to=2-1]
	\arrow["{\beta }", from=1-2, to=2-3]
	\arrow["\gamma"', from=2-1, to=2-3]
\end{tikzcd}\]According to \cite[Proposition 4.3.2.17]{HTT}, for any pair of diagrams
$f,g:K^{\rcone}\to\cal C$ such that $g$ is essentially constant,
the map 
\begin{align*}
\Hom_{\Fun\pr{K^{\rcone},\cal C}}\pr{f,g} & \to\Hom_{\cal C}\pr{f\pr{\infty},g\pr{\infty}}
\end{align*}
is a homotopy equivalence. Therefore, it suffices to show that there
is a diagram $\Delta^{2}\to\cal C$ of the form% https://q.uiver.app/#q=WzAsMyxbMSwwLCJwKFxcaW5mdHkpIl0sWzAsMSwicShcXGluZnR5KSJdLFsyLDEsInIoXFxpbmZ0eSkiXSxbMCwxLCJcXGFscGhhX1xcaW5mdHkiLDJdLFswLDIsIlxcYmV0YV9cXGluZnR5Il0sWzEsMl1d
\[\begin{tikzcd}
	& {p(\infty)} \\
	{q(\infty)} && {r(\infty),}
	\arrow["{\alpha_\infty}"', from=1-2, to=2-1]
	\arrow["{\beta_\infty}", from=1-2, to=2-3]
	\arrow[from=2-1, to=2-3]
\end{tikzcd}\]which is clear because $\alpha_{\infty}$ is an equivalence.
\end{proof}
\begin{lem}
\label{lem:refraction}Let $\frak P$ be a commutative categorical
pattern on an $\infty$-category $\cal D$, let $K$ be a small simplicial
set, and let $X'\to K^{\rcone}\times\cal D$ be a $\frak P$-bundle
over $K^{\rcone}$. There is a map $r:X'_{\natural}\to\pr{X'_{\infty}}_{\natural}$
rendering the diagram % https://q.uiver.app/#q=WzAsNCxbMSwwLCIoWCdfe1xcaW5mdHl9KV9cXG5hdHVyYWwiXSxbMSwxLCJcXG92ZXJsaW5le1xcbWF0aGNhbHtEfX0iXSxbMCwxLCJYJ19cXG5hdHVyYWwiXSxbMCwwLCIoWCdfXFxpbmZ0eSlfXFxuYXR1cmFsIl0sWzAsMV0sWzIsMV0sWzIsMCwiciIsMSx7InN0eWxlIjp7ImJvZHkiOnsibmFtZSI6ImRhc2hlZCJ9fX1dLFszLDAsIlxcb3BlcmF0b3JuYW1le2lkfSJdLFszLDIsIiIsMix7InN0eWxlIjp7InRhaWwiOnsibmFtZSI6Imhvb2siLCJzaWRlIjoidG9wIn19fV1d
\begin{equation}\label{d:refraction}
\begin{tikzcd}
	{(X'_\infty)_\natural} & {(X'_{\infty})_\natural} \\
	{X'_\natural} & {\overline{\mathcal{D}}}
	\arrow["{\operatorname{id}}", from=1-1, to=1-2]
	\arrow[hook, from=1-1, to=2-1]
	\arrow[from=1-2, to=2-2]
	\arrow["r"{description}, dashed, from=2-1, to=1-2]
	\arrow[from=2-1, to=2-2]
\end{tikzcd}
\end{equation}commutative. Moreover, such an $r$ is unique up to homotopy, and
the restriction $r\vert X_{\natural}$ is a refraction map for $p$.
\end{lem}

\begin{proof}
For the existence of the map $r$, it suffices to show that the inclusion
$\pr{X'_{\infty}}_{\natural}\subset X'_{\natural}$ induces a trivial
fibration
\[
\theta:\Map_{\SS_{/K^{\rcone}\times\frak P}^{+}}\pr{X'_{\natural},\pr{K^{\rcone}}^{\sharp}\times\pr{X'_{\infty}}_{\natural}}\xrightarrow{\simeq}\Map_{\SS_{/\frak P}^{+}}\pr{\pr{X'_{\infty}}_{\natural},\pr{X'_{\infty}}_{\natural}}
\]
of Kan complexes. The map $\theta$ is a Kan fibration (because $\SS_{/K^{\rcone}\times\frak P}^{+}$
is a simplicial model category and the inclusion $\pr{X'_{\infty}}_{\natural}\subset X'_{\natural}$
is its cofibration), so it suffices to show that $\theta$ is a homotopy
equivalence. Using Theorem \ref{thm:A24cmain}, we are reduced to
showing that for every pair of diagrams $f,g:K^{\rcone}\to\frak P\-\Fib$
with $g$ essentially constant, the map
\[
\Hom_{\Fun\pr{K^{\rcone},\frak P\-\Fib}}\pr{f,g}\to\Hom_{\frak P\-\Fib}\pr{f\pr{\infty},g\pr{\infty}}
\]
is a homotopy equivalence, which is the content of \cite[Proposition 4.3.2.17]{HTT}.

To complete the proof, we must show that there is some filler of diagram
(\ref{d:refraction}) that restricts to a refraction map of $p$.
For this, let $h':\Delta^{1}\times K^{\rcone}\to K^{\rcone}$ denote
the natural transformation from the identity map to the constant map
at the base point. Since $\SS_{/K^{\rcone}\times\frak P}^{+}$ is
a simplicial model category, the left vertical arrow of the diagram
% https://q.uiver.app/#q=WzAsNSxbMCwwLCIoXFx7MFxcfV5cXHNoYXJwXFx0aW1lcyBYJ19cXG5hdHVyYWwpXFxhbWFsZyBfe1xcezBcXH1eXFxzaGFycCBcXHRpbWVzIChYJ19cXGluZnR5IClfXFxuYXR1cmFsfSgoXFxEZWx0YSBeMSleXFxzaGFycCAgXFx0aW1lcyAoWCdfXFxpbmZ0eSlfXFxuYXR1cmFsKSJdLFswLDEsIihcXERlbHRhXjEpXlxcc2hhcnAgXFx0aW1lcyBYJ19cXG5hdHVyYWwiXSxbMiwwLCJYJ197XFxuYXR1cmFsfSJdLFsyLDEsIihLXlxcdHJpYW5nbGVyaWdodCleXFxzaGFycFxcdGltZXMgXFxvdmVybGluZXtcXG1hdGhjYWx7RH19Il0sWzEsMSwiKFxcRGVsdGFeMSlee1xcc2hhcnB9XFx0aW1lcyAoS15cXHRyaWFuZ2xlcmlnaHQpXlxcc2hhcnBcXHRpbWVzIFxcb3ZlcmxpbmV7XFxtYXRoY2Fse0R9fSJdLFswLDEsIiIsMCx7InN0eWxlIjp7InRhaWwiOnsibmFtZSI6Imhvb2siLCJzaWRlIjoidG9wIn19fV0sWzEsNCwiXFxvcGVyYXRvcm5hbWV7aWR9XFx0aW1lcyBwJyIsMl0sWzQsMywiaCdcXHRpbWVzIFxcb3BlcmF0b3JuYW1le2lkfSIsMl0sWzIsMywicCciXSxbMCwyXSxbMSwyLCJIJyIsMSx7InN0eWxlIjp7ImJvZHkiOnsibmFtZSI6ImRhc2hlZCJ9fX1dXQ==
\[\begin{tikzcd}[column sep = small]
	{(\{0\}^\sharp\times X'_\natural)\amalg _{\{0\}^\sharp \times (X'_\infty )_\natural}((\Delta ^1)^\sharp  \times (X'_\infty)_\natural)} && {X'_{\natural}} \\
	{(\Delta^1)^\sharp \times X'_\natural} & {(\Delta^1)^{\sharp}\times (K^\triangleright)^\sharp\times \overline{\mathcal{D}}} & {(K^\triangleright)^\sharp\times \overline{\mathcal{D}}}
	\arrow[from=1-1, to=1-3]
	\arrow[hook, from=1-1, to=2-1]
	\arrow["{p'}", from=1-3, to=2-3]
	\arrow["{H'}"{description}, dashed, from=2-1, to=1-3]
	\arrow["{\operatorname{id}\times p'}"', from=2-1, to=2-2]
	\arrow["{h'\times \operatorname{id}}"', from=2-2, to=2-3]
\end{tikzcd}\]is a trivial cofibration. (Here the top arrow is the amalgamation
of the identity map of $X'_{\natural}$ and the projection $\pr{\Delta^{1}}^{\sharp}\times\pr{X'_{\infty}}_{\natural}\to\pr{X'_{\infty}}_{\natural}$.)
It follows that there is a dashed filler $H'$ as indicated in the
diagram. The restriction $r=H'\vert\{1\}^{\sharp}\times\pr{X'_{\infty}}_{\natural}$
is a filler of (\ref{d:refraction}) which restricts to a refraction
map, and the proof is complete.
\end{proof}
We now arrive at the proof of Proposition \ref{prop:univ_colim_criterion}:
\begin{proof}
[Proof of Proposition \ref{prop:univ_colim_criterion}]

Choose a retraction $r:X'_{\natural}\to\pr{X'_{\infty}}_{\natural}$
as in Lemma \ref{lem:refraction}. Given a fibration $\pi:Z_{\natural}\to\pr{X'_{\infty}}_{\natural}$
of $\SS_{/\frak P}^{+}$, form pullback squares as in the diagram
% https://q.uiver.app/#q=WzAsNixbMCwyLCJaX3tcXG5hdHVyYWx9Il0sWzEsMiwiKFgnX3tcXGluZnR5fSlfXFxuYXR1cmFsIl0sWzEsMSwiWCdfe1xcbmF0dXJhbH0iXSxbMCwxLCJZJ197XFxuYXR1cmFsfSJdLFsxLDAsIlhfe1xcbmF0dXJhbH0iXSxbMCwwLCJZX3tcXG5hdHVyYWx9Il0sWzAsMSwiXFxwaSIsMl0sWzIsMSwiciJdLFszLDJdLFs0LDJdLFs1LDRdLFs1LDNdLFszLDAsIlxcemV0YSIsMl1d
\[\begin{tikzcd}
	{Y_{\natural}} & {X_{\natural}} \\
	{Y'_{\natural}} & {X'_{\natural}} \\
	{Z_{\natural}} & {(X'_{\infty})_\natural.}
	\arrow[from=1-1, to=1-2]
	\arrow[from=1-1, to=2-1]
	\arrow[from=1-2, to=2-2]
	\arrow[from=2-1, to=2-2]
	\arrow["\zeta"', from=2-1, to=3-1]
	\arrow["r", from=2-2, to=3-2]
	\arrow["\pi"', from=3-1, to=3-2]
\end{tikzcd}\]The map $\pr{Y'_{\infty}}_{\natural}\to Z_{\natural}$ is an isomorphism,
and under this isomorphism, we can identify $\zeta\vert Y_{\natural}$
with the refraction map of $Y'_{\natural}$ (by Lemma \ref{lem:refraction}).
Therefore, by Proposition \ref{prop:A24c.6.8}, we can rephrase condition
(2) as follows:
\begin{itemize}
\item [($2'$)]For every fibration $Z_{\natural}\to\pr{X'_{\infty}}_{\natural}$
in $\SS_{/\frak P}^{+}$, the $\frak P$-bundle $Z_{\natural}\times_{X'_{\infty}}X'_{\natural}$
over $K^{\rcone}$ is classified by a colimit diagram $K^{\rcone}\to\frak P\-\Fib$.
\end{itemize}
Now since $K^{\rcone}$ is weakly contractible, the diagonal functor
$\frak P\-\Fib\to\Fun\pr{K^{\rcone},\frak P\-\Fib}$ is fully faithful,
with essential image consisting of essentially constant functors (Remark
\ref{rem:ess_cst}). It follows from Theorem \ref{thm:A24cmain} that
the functor $K^{\rcone}\times-:\frak P\-\Fib\to\frak P\-\Bund\pr{K^{\rcone}}$
is also fully faithful. Combining this observation with \cite[Proposition 4.2.4.1]{HTT},
we can further rephrase (2$'$) as follows:
\begin{itemize}
\item [($2''$)]For every pullback diagram % https://q.uiver.app/#q=WzAsNCxbMCwxLCJCIl0sWzEsMSwiKEtee1xcdHJpYW5nbGVyaWdodCB9KV5cXHNoYXJwXFx0aW1lcyAoWCdfe1xcaW5mdHl9KV9cXG5hdHVyYWwiXSxbMSwwLCJYJ197XFxuYXR1cmFsfSJdLFswLDAsIkEiXSxbMCwxXSxbMiwxXSxbMywyXSxbMywwXV0=
\[\begin{tikzcd}
	A & {X'_{\natural}} \\
	B & {(K^{\triangleright })^\sharp\times (X'_{\infty})_\natural}
	\arrow[from=1-1, to=1-2]
	\arrow[from=1-1, to=2-1]
	\arrow[from=1-2, to=2-2]
	\arrow[from=2-1, to=2-2]
\end{tikzcd}\]in $\frak P\-\Bund\pr{K^{\rcone}}$ (where the right hand map is induced
by $r$), if $B$ is classified by an essentially constant diagram
$K^{\rcone}\to\frak P\-\Fib$, then $A$ is a colimit diagram.
\end{itemize}
Lemma \ref{lem:univ_colim} and Theorem \ref{thm:A24cmain} now show
that (2$''$) is equivalent to (1), and the proof is complete.
\end{proof}

\subsection{\label{subsec:univ_weq_op}Universal Weak Equivalence of \texorpdfstring{$\protect\SS_{/\protect\frak{Op}}^{+}$}{Preoperads}}

Let $K$ be a Kan complex. As we saw in Remark \ref{rem:family_vs_bundle},
if $\cal E\to K\times N\pr{\Fin_{\ast}}$ is a $K$-bundle of $\infty$-operads,
then it is a $K$-family of $\infty$-operads (in particular, a generalized
$\infty$-operad \cite[Proposition 2.3.2.11]{HA}) and the marked
edges of $\cal E_{\natural}$ are precisely the inert morphisms. This,
together with the universal colimit criterion we established in Subsection
\ref{subsec:univ_weq} (Proposition \ref{prop:univ_colim_criterion})
motivates the following question: Let $\cal A^{\t}\to\cal B^{\t}$
be a morphism of generalized $\infty$-operads, where $\cal B^{\t}$
is an $\infty$-operad. Let $\cal A^{\t,\natural},\cal B^{\t,\natural}$
denote the marked simplicial sets obtained from $\cal A^{\t}$ and
$\cal B^{\t}$ by marking the inert maps. When is the map $\cal A^{\t,\natural}\to\cal B^{\t,\natural}$
a universal weak equivalence of $\SS_{/\frak{Op}}^{+}$?

A more general question has been posed by Lurie in \cite[$\S$2.3.3]{HA}
for (non-universal) weak equivalences of $\SS_{/\frak{Op}}^{+}$.
While this was not explicitly stated by him, we will see that his
answer in fact accommodates universal weak equivalences of $\SS_{/\frak{Op}}^{+}$
(Proposition \ref{prop:universal_assembly}). 

We start by recalling the following theorem, which is a special case
of \cite[Theorem 2.3.3.23]{HA}.
\begin{thm}
\label{thm:2.3.3.23}Let $f:\cal A^{\t}\to\cal B^{\t}$ be a morphism
of generalized $\infty$-operads, where $\cal B^{\t}$ is an $\infty$-operad.
Suppose that $f$ satisfies the following conditions:
\begin{enumerate}
\item The functor $\cal A\to\cal B$ is a categorical equivalence.
\item For each object $B\in\cal B^{\t}$, the inclusion $\pr{\cal A_{B/}^{\t}}'\subset\cal A_{B/}^{\t}$
is initial, where $\pr{\cal A_{B/}^{\t}}^{\p}\subset\cal A_{B/}^{\t}$
denotes the full subcategory spanned by the objects $\pr{A,\alpha:B\to f\pr A}$
such that $\alpha$ is inert.
\end{enumerate}
Then the map $\cal A^{\t,\natural}\to\cal B^{\t,\natural}$ is a weak
equivalence of $\SS_{/\frak{Op}}^{+}$.
\end{thm}

Condition (2) of Theorem \ref{thm:2.3.3.23} admits the following
reformulation, which follows from the argument of \cite[2.3.3.11]{HA}.
\begin{prop}
\label{prop:2.3.3.11}Let $f:\cal A^{\t}\to\cal B^{\t}$ be a map
of generalized $\infty$-operads. The following conditions are equivalent:
\begin{enumerate}
\item For each object $B\in\cal B^{\t}$, the inclusion $\pr{\cal A_{B/}^{\t}}'\subset\cal A_{B/}^{\t}$
is initial, where $\pr{\cal A_{B/}^{\t}}'$ is defined as in Theorem
\ref{thm:2.3.3.23}.
\item For each object $A\in\cal A^{\t}$, the homotopy fibers (in the Joyal
model structure) of the functor
\[
\pr{\cal A_{\act}^{\t}}_{/A}\to\pr{\cal B_{\act}^{\t}}_{/f\pr A}
\]
are weakly contractible.\footnote{Here $\pr{\cal A_{\act}^{\t}}_{/A}$ denotes the slice of $\cal A_{\act}^{\t}$
with respect to $A$, not the fiber product $\cal A_{\act}^{\t}\times_{\cal A^{\t}}\cal A_{/A}^{\t}$.}
\end{enumerate}
\end{prop}

Combining Theorem \ref{thm:2.3.3.23} and Proposition \ref{prop:2.3.3.11},
we obtain the following criterion for universal weak equivalences
in $\SS_{/\frak{Op}}^{+}$.
\begin{prop}
\label{prop:universal_assembly}Let $f:\cal A^{\t}\to\cal B^{\t}$
be a morphism of generalized $\infty$-operads, where $\cal B^{\t}$
is an $\infty$-operad. Suppose that the functor $\cal A\to\cal B$
is a categorical equivalence and that for each object $A\in\cal A^{\t}$,
the homotopy fibers (in the Joyal model structure) of the functor
\[
\pr{\cal A_{\act}^{\t}}_{/A}\to\pr{\cal B_{\act}^{\t}}_{/f\pr A}
\]
are weakly contractible. Then the map $f:\cal A^{\t,\natural}\to\cal B^{\t,\natural}$
is a universal weak equivalence of $\SS_{/\frak{Op}}^{+}$.
\end{prop}

\begin{proof}
Let $\cal D^{\t}\to\cal B^{\t}$ be a fibration of $\infty$-operads
and set $\cal C^{\t}=\cal A^{\t}\times_{\cal B^{\t}}\cal D^{\t}$.
We must show that the map $g:\cal C^{\t,\natural}\to\cal D^{\t,\natural}$
is a weak equivalence of $\SS_{/\frak{Op}}^{+}$. By Theorem \ref{thm:2.3.3.23}
and Proposition \ref{prop:2.3.3.11}, it suffices to show that, for
each object $C\in\cal C^{\t}$, the functor
\[
\theta:\pr{\cal C_{\act}^{\t}}_{/C}\to\pr{\cal D_{\act}^{\t}}_{/g\pr C}
\]
has weakly contractible homotopy fibers. If $A\in\cal A^{\t}$ denotes
the image of $C$, then the map $\theta$ is a pullback of the functor
$\pr{\cal A_{\act}^{\t}}_{/A}\to\pr{\cal B_{\act}^{\t}}_{/f\pr A}$.
Since the latter has weakly contractible homotopy fibers, we are done.
\end{proof}

\subsection{\label{subsec:key}Proof of Lemma \ref{lem:key}}

In this subsection, we will prove Lemma \ref{lem:key}, the main result
of this section, by using results in Subsections \ref{subsec:univ_weq}
and \ref{subsec:univ_weq_op}.

We begin with a lemma.
\begin{lem}
\label{lem:K^=00005Camalg_approx}Let $K$ be a Kan complex, and let
$f:\cal O^{\t}\to K^{\amalg}$ be a morphism of generalized $\infty$-operads.
Suppose that $f$ satisfies the following conditions:
\begin{enumerate}
\item The $\infty$-category $\cal O$ is a Kan complex, and the map $\cal O\to K$
is a homotopy equivalence of Kan complexes. 
\item The functor $\cal O^{\t}\to N\pr{\Fin_{\ast}}$ is conservative.
\item For each $n\geq0$ and each object $X\in\cal O_{\inp n}^{\t}$, the
map $\pr{\pr{\cal O_{\act}^{\t}}_{/X}}^{\simeq}\to\pr{\pr{N\pr{\Fin_{\ast}}_{\act}}_{/\inp n}}^{\simeq}$
is a homotopy equivalence.
\end{enumerate}
Then $f$ is a universal weak equivalence of $\SS_{/\frak{Op}}^{+}$.
\end{lem}

\begin{proof}
According to Proposition \ref{prop:universal_assembly}, it will suffice
to show that, for each object $X\in\cal O^{\t}$, the functor
\[
\pr{\cal O_{\act}^{\t}}_{/X}\to\pr{K_{\act}^{\amalg}}_{/f\pr X}
\]
has weakly contractible homotopy fibers. By condition (2), this functor
is conservative. Therefore, it suffices to show that the map
\[
\pr{\pr{\cal O_{\act}^{\t}}_{/X}}^{\simeq}\to\pr{\pr{K_{\act}^{\amalg}}_{/f\pr X}}^{\simeq}
\]
is a homotopy equivalence of Kan complexes. Condition (3) now shows
that this is equivalent to the condition that the map
\[
\pi:\pr{\pr{K_{\act}^{\amalg}}_{/f\pr X}}^{\simeq}\to\pr{\pr{N\pr{\Fin_{\ast}}_{\act}}_{/\inp n}}^{\simeq}
\]
be a homotopy equivalence, where $\inp n\in\Fin_{\ast}$ denotes the
image of $X$. Since $\pi$ is a Kan fibration, it suffices to show
that its fibers are contractible. But the fibers of $\pi$ are products
of simplicial sets of the form $K_{/v}$, where $v$ is some vertex
of $K$. In particular, the fibers of $\pi$ are contractible. The
proof is now complete.
\end{proof}
\begin{proof}
[Proof of Lemma \ref{lem:key}]Set $K'=\cal D_{/K}$ and let $\cal O_{\natural}^{\t}\to\pr{K^{\p\rcone}}^{\sharp}\times N\pr{\Fin_{\ast}}_{\natural}$
be an $\frak{Op}$-bundle classified by the composite
\[
K^{\p\rcone}=\pr{\cal D_{/K}}^{\rcone}\to\cal S\xrightarrow{\pr -^{\amalg}}\Op_{\infty}.
\]
Let $\chi:\cal O_{\natural}^{\t}\times_{\pr{K^{\p\rcone}}^{\sharp}}\pr{K^{\p}}^{\sharp}\to\cal O_{\natural}^{\t}\times_{\pr{K^{\p\rcone}}^{\sharp}}\{\infty\}^{\sharp}$
denote the refraction map. By Proposition \ref{prop:univ_colim_criterion},
it will suffice to show that the map $\chi$ is a universal weak equivalence
of $\SS_{/\frak{Op}}^{+}$. According to Lemma \ref{lem:K^=00005Camalg_approx},
it will suffice to prove the following:
\begin{enumerate}
\item The simplicial set $\cal O\times_{\pr{K^{\p\rcone}}^{\rcone}}K'$
is a Kan complex, and the refraction map $\chi$ restricts to a homotopy
equivalence 
\[
\alpha:\cal O\times_{\pr{K^{\p\rcone}}^{\rcone}}K'\xrightarrow{\simeq}\cal O\times_{\pr{K^{\p\rcone}}}\{\infty\}
\]
of Kan complexes.
\item The functor $\cal O^{\t}\times_{K^{\p\rcone}}K'\to N\pr{\Fin_{\ast}}$
is conservative.
\item For each $n\geq0$ and each object $X\in\cal O_{\inp n}^{\t}$, the
map 
\[
\gamma:\pr{\pr{\cal O^{\t}\times_{K^{\p\rcone}}K'}_{/X}}^{\simeq}\to\pr{\pr{N\pr{\Fin_{\ast}}}_{/\inp n}}^{\simeq}
\]
 is a trivial fibration.
\end{enumerate}

We begin with (1). By the definition of bundles, the functor $\cal O\to K^{\p\rcone}$
is a cocartesian fibration. Moreover, the fibers of the map $\cal O\to K^{\p\rcone}$
are Kan complexes (Remark \ref{rem:A24cmain}), so it is a left fibration.
Since $K'$ is a Kan complex, it follows that $\cal O\times_{K^{\p\rcone}}K'$
is a Kan complex. By Remark \ref{rem:A24cmain}, we can identify $\alpha$
with the refraction map of the $K^{\p\rcone}$-bundle classified by
the composite
\[
K^{\p\rcone}=\pr{\cal D_{/K}}^{\rcone}\to\cal S\xrightarrow{\pr -^{\amalg}}\Op_{\infty}\xrightarrow{\pr -_{\inp 1}}\Cat_{\infty},
\]
where $\pr -_{\inp 1}$ denotes the functor which assigns to each
$\infty$-operad $\cal O^{\t}$ the fiber $\cal O_{\inp 1}^{\t}=\cal O^{\t}\times_{N\pr{\Fin_{\ast}}}\{\inp 1\}$.
But the composite $\pr -_{\inp 1}\circ\pr -^{\amalg}:\cal S\to\Cat_{\infty}$
is just the inclusion, so we are reduced to showing that the inclusion
$\cal S\to\Cat_{\infty}$ is a left Kan extension of its restriction
to $\cal D$. This is clear, because the functor $\cal S\to\Cat_{\infty}$
is a left adjoint and the identity functor of $\cal S$ is a left
Kan extension of its restriction to $\cal D$ \cite[Lemma 5.1.5.3]{HTT}. 

Next, we prove (2). Since the $K'$-bundle $\cal O^{\t}\times_{K^{\p\rcone}}K'$
is equivalent to the terminal $K'$-bundle $K'\times N\pr{\Fin_{\ast}}$
(because the constant diagram at $N\pr{\Fin_{\ast}}$ is a terminal
object of $\Fun\pr{K',\Op_{\infty}}$), it suffices to show that the
functor $K'\times N\pr{\Fin_{\ast}}\to N\pr{\Fin_{\ast}}$ is conservative,
which is obvious.

Finally, we prove (3). As in (2), it will suffice to show that, for
each object $\pr{v,\inp n}\in K'\times N\pr{\Fin_{\ast}}$, the map
\[
\pr{\pr{K'\times N\pr{\Fin_{\ast}}}_{/\pr{v,\inp n}}}^{\simeq}=K'_{/v}\times\pr{N\pr{\Fin_{\ast}}_{/\inp n}}^{\simeq}\to\pr{N\pr{\Fin_{\ast}}_{/\inp n}}^{\simeq}
\]
is a trivial fibration. This is clear, since $K'_{/v}$ is a contractible
Kan complex. The proof is now complete.
\end{proof}

\newcommand{\etalchar}[1]{$^{#1}$}
\providecommand{\bysame}{\leavevmode\hbox to3em{\hrulefill}\thinspace}
\providecommand{\MR}{\relax\ifhmode\unskip\space\fi MR }
\providecommand{\MRhref}[2]{%
  \href{http://www.ams.org/mathscinet-getitem?mr=#1}{#2}
}
\providecommand{\href}[2]{#2}


\begin{thebibliography}{DHL{\etalchar{+}}23}

\bibitem[AF15]{FHTM}
David Ayala and John Francis, \emph{Factorization homology of topological
  manifolds}, J. Topol. \textbf{8} (2015), no.~4, 1045--1084. \MR{3431668}

\bibitem[Ara24a]{A24b}
Kensuke Arakawa, \emph{A context for manifold calculus},
  \url{https://arxiv.org/abs/2403.03321}, 2024.

\bibitem[Ara24b]{A24c}
\bysame, \emph{The {G}rothendieck construction for $\infty$-categories fibered
  over categorical patterns}, \url{https://arxiv.org/abs/2404.01025}, 2024.

\bibitem[BV73]{BV73}
J.~M. Boardman and R.~M. Vogt, \emph{Homotopy invariant algebraic structures on
  topological spaces}, Lecture Notes in Mathematics, vol. Vol. 347,
  Springer-Verlag, Berlin-New York, 1973. \MR{420609}

\bibitem[CG17]{FAQFT_1}
Kevin Costello and Owen Gwilliam, \emph{Factorization algebras in quantum field
  theory. {V}ol. 1}, New Mathematical Monographs, vol.~31, Cambridge University
  Press, Cambridge, 2017. \MR{3586504}

\bibitem[CG21]{FAQFT_2}
\bysame, \emph{Factorization algebras in quantum field theory. {V}ol. 2}, New
  Mathematical Monographs, vol.~41, Cambridge University Press, Cambridge,
  2021. \MR{4300181}

\bibitem[DHL{\etalchar{+}}23]{DHLSW23}
Sanath Devalapurkar, Jeremy Hahn, Tyler Lawson, Andrew Senger, and Dylan
  Wilson, \emph{Examples of disk algebras},
  \url{https://arxiv.org/abs/2302.11702}, 2023.

\bibitem[Fre20]{Benoit20}
Benoit Fresse, \emph{Little discs operads, graph complexes and
  {G}rothendieck-{T}eichm\"{u}ller groups}, Handbook of homotopy theory, CRC
  Press/Chapman Hall Handb. Math. Ser., CRC Press, Boca Raton, FL, [2020]
  \copyright 2020, pp.~405--441. \MR{4197991}

\bibitem[Get94]{Getz94}
E.~Getzler, \emph{Batalin-{V}ilkovisky algebras and two-dimensional topological
  field theories}, Comm. Math. Phys. \textbf{159} (1994), no.~2, 265--285.
  \MR{1256989}

\bibitem[GHN17]{GHN17}
David Gepner, Rune Haugseng, and Thomas Nikolaus, \emph{Lax colimits and free
  fibrations in {$\infty$}-categories}, Doc. Math. \textbf{22} (2017),
  1225--1266. \MR{3690268}

\bibitem[HKK22]{HKK22}
Geoffroy Horel, Manuel Krannich, and Alexander Kupers, \emph{Two remarks on
  spaces of maps between operads of little cubes},
  \url{https://arxiv.org/abs/2211.00908}, 2022.

\bibitem[Joy02]{Joyal_qcat_Kan}
A.~Joyal, \emph{Quasi-categories and {K}an complexes}, vol. 175, 2002, Special
  volume celebrating the 70th birthday of Professor Max Kelly, pp.~207--222.
  \MR{1935979}

\bibitem[Kis64]{Kister_MFB}
J.~M. Kister, \emph{Microbundles are fibre bundles}, Ann. of Math. (2)
  \textbf{80} (1964), 190--199. \MR{180986}

\bibitem[KSW24]{KSW24}
Eilind Karlsson, Claudia~I. Scheimbauer, and Tashi Walde, \emph{Assembly of
  constructible factorization algebras},
  \url{https://arxiv.org/abs/2403.19472}, 2024.

\bibitem[Lur09a]{DAGVI}
Jacob Lurie, \emph{Derived algebraic geometry {VI}: {$E_k$} algebras},
  \url{https://arxiv.org/abs/0911.0018}, 2009.

\bibitem[Lur09b]{HTT}
Jacob Lurie, \emph{Higher topos theory}, Annals of Mathematics Studies, vol.
  170, Princeton University Press, Princeton, NJ, 2009. \MR{2522659}

\bibitem[Lur17]{HA}
J.~Lurie, \emph{Higher algebra},
  \url{https://www.math.ias.edu/~lurie/papers/HA.pdf}, 2017.

\bibitem[Lur24]{kerodon}
Jacob Lurie, \emph{Kerodon}, \url{https://kerodon.net}, 2024.

\bibitem[Mar99]{Markl99}
Martin Markl, \emph{A compactification of the real configuration space as an
  operadic completion}, J. Algebra \textbf{215} (1999), no.~1, 185--204.
  \MR{1684178}

\bibitem[Mat17]{Matsu17}
Takuo Matsuoka, \emph{Descent properties of topological chiral homology},
  M\"{u}nster J. Math. \textbf{10} (2017), no.~1, 83--118. \MR{3624103}

\bibitem[May72]{MayOperad}
J.~P. May, \emph{The geometry of iterated loop spaces}, Lecture Notes in
  Mathematics, vol. Vol. 271, Springer-Verlag, Berlin-New York, 1972.
  \MR{420610}

\bibitem[Wah01]{Wahl01}
Nathalie Wahl, \emph{Ribbon braids and related operads}, Ph.D. thesis,
  University of Oxford, 2001.

\end{thebibliography}
\end{document}